\documentclass[preprint,12pt]{elsarticle}
\usepackage{amssymb}
\usepackage{amsmath}
\usepackage{amsthm}
\usepackage{cases}
\usepackage{amscd}
\usepackage{amsfonts}
\usepackage{hyperref}
\usepackage{enumitem}
\usepackage{graphicx}  
\usepackage{comment}

\usepackage{subfigure}
\usepackage{algorithm}
\usepackage{algpseudocode}
\renewcommand{\algorithmicrequire}

\usepackage{cleveref}
\usepackage{color}
\usepackage{datetime}

\allowdisplaybreaks

\makeatletter
\newcommand\figcaption{\def\@captype{figure}\caption}
\newcommand\tabcaption{\def\@captype{table}\caption}
\makeatother

\usepackage{geometry}
\usepackage{algorithm}
\usepackage{multirow}
\usepackage{mathrsfs}

\newtheorem{remark}{Remark}
\newtheorem{theorem}{Theorem}
\newtheorem{Assumption}{Assumption}
\newtheorem{lemma}{Lemma}
\newtheorem{definition}{Definition}
\newtheorem{corollary}{Corollary}
\newcommand{\beq}{\begin{equation}}
	\newcommand{\eeq}{\end{equation}}
\newcommand{\bea}{\begin{eqnarray}}
	\newcommand{\eea}{\end{eqnarray}}
\newcommand{\beas}{\begin{eqnarray*}}
	\newcommand{\eeas}{\end{eqnarray*}}

\begin{document}

\begin{frontmatter}



\title{Inverse source problems for the stochastic wave equations\tnoteref{1}}
\tnotetext[1]{The work was supported by NSFC Project (12431014).}

\author[a]{Yunqing Huang} 
\ead{huangyq@xtu.edu.cn}
\author[b]{Shihan Zhang\corref{cor1}}
\ead{shihanzhang@smail.xtu.edu.cn}
\cortext[cor1]{Corresponding author}
\affiliation[a]{organization={National Center for Applied Mathematics in Hunan, Key Laboratory of Intelligent Computing $\&$ Information Processing of Ministry of Education, Xiangtan University},
            addressline={}, 
            city={Xiangtan},
            postcode={411105}, 
            state={Hunan},
            country={China}}
\affiliation[b]{organization={School of Mathematics and Computational Science, Xiangtan University},
	addressline={}, 
	city={Xiangtan},
	postcode={411105}, 
	state={Hunan},
	country={China}}

\begin{abstract}
To address the ill-posedness of the inverse source problem for the one-dimensional stochastic Helmholtz equations without attenuation,  this study develops a novel computational framework designed to mitigate this inherent challenge at the numerical implementation level. For the stochastic wave equation driven by a finite-jump Lévy process (assuming that its jump amplitude obeys a Gaussian distribution and the jump time interval obeys a Poisson distribution), this paper firstly establish the existence of a mild solution to its direct problem satisfying a particular stability estimate. Building upon these theoretical foundations, we further investigate the well-posedness of the inverse problem and develop a methodology to reconstruct the unknown source terms $f$ and $g$ using the data of the wave field at the final time point $u(x,T)$.
This work not only provides rigorous theoretical analysis and effective numerical schemes for solving inverse source problems in these two specific classes of stochastic wave equations, but also offers new perspectives and methodological approaches for addressing a broader range of wave propagation inverse problems characterized by non-Gaussian stochastic properties. The proposed framework demonstrates significant relevance for characterizing physical phenomena influenced by jump-type stochastic perturbations, offering promising applications in diverse domains including but not limited to seismic wave propagation analysis and financial market volatility modeling.
\end{abstract}




\end{frontmatter}


\section{Introduction}
The study of the inverse source problem of wave equations is an active and important topic, which essentially utilizes the relevant information measured by wavefield to recover unknown sources. As is well known, there are countless sources of radiation fields that completely disappear outside their supporting volume, so there is no unique solution to the inverse source problem. cf.\cite{devaney2003nonuniqueness}.
By adding any of these non radiative sources to any given solution, an infinite number of solutions can be obtained. Assuming you have an observed electromagnetic field, you want to infer the position and strength of the source. However, in reality, there may be multiple combinations of different charge distributions that can generate the same electromagnetic field. Moreover, these different charge distributions may have different boundary conditions or dielectric properties (which may be "non radiative sources"), and they do not directly generate electromagnetic waves, but by changing the properties or boundary conditions of the medium, the final electromagnetic field performance will be almost the same. In this case, there is no unique solution to the inverse source problem, as you can obtain the same radiation field under different combinations of sources and non radiation sources. Therefore, the inverse source problem is usually ill-posed, and obtaining a method to ensure its posedness is an urgent problem that we need to solve.

The inverse source problem of wave equation in magnetoencephalography\cite{bao2002inverse}, ultrasonics\cite{devaney1983inverse} and photoacoustic tomography\cite{anastasio2007application}, particularly for modeling wave propagation phenomena with random disturbances. In \cite{bao2010multi, bao2011numerical, bao2011inverse}, the authors explored the use of multi-frequency data to ensure the uniqueness and stability of the inverse source problem for the acoustic wave equation, establishing stability estimates based on the radiation field outside the source volume at a set of frequencies. Following these works, significant research has been conducted to improve the stability of various wave equation inverse source problems \cite{bao2020stability, cheng2016increasing, entekhabi2020increasing}. With increasing interest in modeling stochastic perturbations and uncertainties, stochastic wave equations have become a major area of research. The introduction of randomness and uncertainty parameters has transformed deterministic inverse source problems into more complex stochastic inverse source problems. These parameters are particularly useful in modeling unpredictable environmental conditions, incomplete system information, and uncertainties arising from measurement noise \cite{borcea2002imaging}. The study of stochastic inverse problems dates back to \cite{devaney1979inverse}, where the autocorrelation function of the stochastic source was shown to be uniquely determined by the autocorrelation of the radiation field outside the source region. To better characterize system uncertainties, random sources have been incorporated into mathematical modeling, as discussed in \cite{bao2016inverse, bao2017inverse, devaney1979inverse}. However, the nonlinearity and ill-posedness of the inverse problem for random sources present significant challenges, which arise from: a) the need to describe random sources using statistical quantities (e.g., mean, variance, and covariance) rather than deterministic functions; and b) the roughness of random sources, which makes pointwise definition impossible and requires distributional interpretation. As such, only the statistical properties of the random source can be reconstructed from the statistical data of the wave field. However, the suitability of statistical data from the wave field and the random source is heavily dependent on the specific form of the random source, further complicating the inverse problem.

The stochastic Helmholtz equation, as the fundamental equation for steady-state wave propagation, plays a critical role in fields such as acoustics, electromagnetism, and quantum mechanics. In \cite{lassas2008inverse}, a novel inverse stochastic source model, the generalized microlocal isotropic Gaussian random field, was proposed, characterized by a covariance operator modeled as a classical pseudodifferential operator. In \cite{li2017stability}, the stability of the inverse stochastic source problem for a one-dimensional Helmholtz equation driven by white noise in a homogeneous medium was analyzed. In \cite{guanstochastic}, the Helmholtz equation was studied using a stochastic Galerkin method combined with a generalized polynomial chaos (gPC) expansion. The mean and variance of the source were computed using the zero-order and higher-order coefficients of the gPC expansion of the boundary data, and reconstruction was carried out using the inverse sine transform. The inverse problem of a one-dimensional stochastic Helmholtz equation with attenuation, where the source term is a generalized microlocal isotropic Gaussian random field, was addressed in\cite{li2020inverse}, yielding pointwise stable reconstruction of the source term. However, in the case without attenuation, the inverse problem remains ill-posed, prompting further exploration into its resolution and improvement.

In the context of inverse source problems for stochastic wave equations driven by random sources, \cite{li2021inverse} examined the inverse problem for the Helmholtz equation driven by fractional Gaussian fields. In this study, the second-order moments of the wave field in the high-frequency limit uniquely determined the strength function of the random source in two-dimensional and three-dimensional cases. In \cite{feng2021inverse}, the authors investigated stochastic wave equations driven by fractional Brownian motion, employing statistical methods and a truncated regularity approach to reconstruct the source term from final-time data. In \cite{li2022inverse}, the analysis focused on the far-field regime, where the covariance and correlation operators of the source were recovered, providing a unified framework for stochastic acoustic, electromagnetic, elastic, and biharmonic waves. However, to date, no studies have explored the inverse source problem for stochastic wave equations driven by Lévy processes. Lévy processes, with their non-Gaussian characteristics and ability to model jump-type stochastic disturbances, are particularly well-suited for describing real-world phenomena such as seismic wave propagation and financial market fluctuations, which makes them particularly suitable for applications in seismic engineering, acoustical design, and digital signal processing. Moreover, these studies offer valuable insights into the development of stochastic partial differential equation theory, especially in dealing with non-Gaussian random sources. This research extends to broader applications in stochastic systems. Since Lévy processes with finite jumps are more common in real-world systems, this paper addresses the inverse source problem of stochastic wave equations driven by finite-jump Lévy processes, providing several results.

This paper investigates the inverse source problem for one-dimensional stochastic Helmholtz equations without attenuation and stochastic wave equations driven by Lévy processes. The main contributions of this paper are as follows:

We propose a new method for mitigating the ill-posedness of the inverse source problem of one-dimensional stochastic Helmholtz equations without attenuation from a computational perspective.

For stochastic wave equations driven by finite-jump Lévy processes, we demonstrate that the forward problem admits a well-defined mild solution that satisfies certain stability estimate. Additionally, we provide a reconstruction method for the source terms $f$ and $g$ from the final-time data of the wave field $u(x,T)$.

For convenience, it is noteworthy that we assume the jump amplitudes of the finite-jump Lévy process follow a Gaussian distribution, while the jump time intervals follow a Poisson distribution, which is a standard stochastic model. Although the source terms considered in this paper are relatively simple, in practice, jump processes are often divided into large and small jumps, which present greater challenges. This paper serves as an initial step toward understanding these complex cases.

The structure of this paper is as follows. In the following Section 2, we presents the regularity of the fundamental solution of the one-dimensional stochastic Helmholtz equation without attenuation and introduces preliminary knowledge on the stochastic wave equation driven by finite-jump Lévy processes.
Section 3 introduces how to combine efficient numerical methods to address the ill-posedness of the inverse problem for the one-dimensional stochastic Helmholtz equation without attenuation.
Section 4 analyzes the direct problem for the stochastic wave equation driven by finite-jump Lévy processes and provides a reconstruction scheme for the source term function in the inverse problem.
The paper concludes with the findings and outlines directions for future research on the inverse source problem of stochastic wave equations. Acknowledgments are also provided for funding support and constructive suggestions.
\section{Preliminary}
In this section, we will introduce the fundamental solutions and stochastic source forms of the one-dimensional stochastic Helmholtz equation without attenuation, as well as some of their properties. In addition, we will cover related knowledge about finite-jump Lévy processes.

\subsection{One-dimensional stochastic Helmholtz equation without attenuation}
The study of the background of one-dimensional stochastic Helmholtz equations without attenuation involves a variety of problems in physics, mathematics, and engineering. These equations have a wide range of applications, particularly in the modeling of wave problems and the solution of inverse problems.

One-dimensional stochastic Helmholtz equation
\[ \Delta u + (\nu^2 + i \nu \sigma) u = f. \]
where the wavenumber $\nu > 0$ is a parameter describing the propagation characteristics of waves in space, typically defined as the number of wave periods per unit length. The attenuation parameter $\sigma \geq 0$ describes the electrical conductivity of the medium. $u$ denotes the wave field, and $f$ is defined on the complete probability space $(\Omega,\mathcal{F}
,\mathbb{P})$ and denotes a random function assumed to be compactly supported on the bounded domain $M=[0,1]$.

We consider the case where the attenuation parameter $\sigma$ is 0, i.e.,
\begin{equation}
	\Delta u + \nu^2 u = f.
	\label{helmholtz non attenuation}
\end{equation}

The wave field $u$ satisfies the outward radiation boundary condition
\begin{equation}
	u'(0)+i\nu u(0)=0,~~u'(1)-i\nu u(1)=0.
	\label{boundary condition}
\end{equation}
at $x=0$, it behaves as a left-propagating wave; at $x=1$, it behaves as a right-propagating wave. Namely, the wave can only leave the domain and cannot enter the domain or produce reflection.

The random source function $f$ satisfies the following assumption.
\begin{Assumption}
	Let $f$ be a real-valued, centrally symmetric, locally isotropic Gaussian random field of order $-s$ with compact support in $M \subset \mathbb{R}^d$, i.e., the covariance operator $C_f$ of $f$ is a pseudodifferential operator with principal symbol $\mu(x)|\xi|^{-s}$, where $\mu \in C_0^{\infty}(M)$ and $\mu \geq 0$, and $M = (0,1)$.
	\label{Ass:1}
\end{Assumption}

\begin{remark}
	$C_{0}^{\infty}(M)$ is defined as the set of all compactly supported smooth functions (smooth functions) on the domain M.
	The statistical properties of a random field (such as mean and covariance) are entirely determined by the Gaussian distribution; micro-local isotropy refers to the fact that the statistical properties of a random field are isotropic at small scales, i.e., its properties are the same in any direction; the order of a random field is $-s$, meaning that its smoothness or roughness is controlled by the parameter $s$. Specifically, larger values of $s$ correspond to smoother random fields, while smaller values of $s$ correspond to rougher random fields.  
\end{remark}
Combining Assumption \ref{Ass:1}, the random source $f$ can be expressed as
\[ f(x)=\sqrt{\mu(x)}(-\Delta)^{-\frac{s}{4}}W_{x}^{'}. \]
where $W_{x}$ denotes a one-dimensional spatial Wiener process, and $W_{x}^{'}$ denotes spatial white noise. $(-\Delta)^{-\frac{s}{4}}$ is the fractional Laplace operator defined on $\mathbb{R}^d$, which is given by
\[ (-\Delta)^{\alpha}u=\mathcal{F}^{-1}[|\xi|^{2\alpha}\mathcal{F}[u](\xi)], \quad \alpha\in \mathbb{R}. \]

Let $f:\Omega\rightarrow\mathcal{G^{'}}$ be measurable. For any $\phi\in\mathcal{G}$, let the mapping $\omega\mapsto\langle f(\omega),\phi\rangle$ be defined, where $\mathcal{G^{'}}$ is the dual space of $\mathcal{G}$. $C_f:\mathcal{G}\rightarrow\mathcal{G^{'}}$ is given as follows
\[ \langle \varphi , C_f\psi \rangle = \mathbb{E} \left[ \langle f, \varphi \rangle \langle f, \psi \rangle \right], \quad \forall \varphi , \psi \in \mathcal{G}. \]
where $\langle\cdotp,\cdotp\rangle$ denotes the dual product. According to the Schwartz kernel theorem, there exists a unique kernel $K_f$ for $C_f$ such that
\[ \langle \varphi,C_f \psi \rangle = \int_{\mathbb{R}^d} \int_{\mathbb{R}^d} K_f(x, y) \varphi(x) \psi(y) \, dx \, dy. \]
therefore, we obtain the following form of the Schwartz kernel
\[ K_f(x, y)=\mathbb{E} \left[f(x),f(y)\right]. \]

$\mu$ denotes the micro-correlation strength of the random field $f$. According to Assumption \ref{Ass:1}, $C_f$ satisfies
\[ (C_f \psi)(x) = \frac{1}{(2\pi)^d} \int_{\mathbb{R}^d} e^{ix \cdot \xi} \, c(x, \xi) \hat{\psi}(\xi) \, d\xi, \quad   \forall \psi \in \mathcal{G}. \]
where $c(x, \xi)$ has a dominant term $\mu(x)|\xi|^{-s}$ and
\[ \hat{\psi}(\xi) = \mathcal{F}[\psi](\xi) = \int_{\mathbb{R}^d} e^{-ix \cdot \xi} \psi(x) \, dx. \]
is the Fourier transform form of $\psi$.
\begin{align*}
	\langle \varphi, C_f \psi \rangle &= \int_{\mathbb{R}^d} \varphi(x) \left[ \frac{1}{(2\pi)^d} \int_{\mathbb{R}^d} e^{ix \cdot \xi} c(x, \xi) \hat{\psi}(\xi) \, d\xi \right] dx\\
	&= \frac{1}{(2\pi)^d} \int_{\mathbb{R}^d} \varphi(x) \int_{\mathbb{R}^d} e^{ix \cdot \xi} c(x, \xi) \left[ \int_{\mathbb{R}^d} e^{-iy \cdot \xi} \psi(y) \, dy \right] d\xi dx\\
	&= \int_{\mathbb{R}^d} \int_{\mathbb{R}^d} \frac{1}{(2\pi)^d} e^{i(x - y) \cdot \xi} c(x, \xi) \varphi(x) \psi(y) \, dx \, dy.
\end{align*}
Consequently,
\[ K_f(x, y) = \frac{1}{(2\pi)^d} \int_{\mathbb{R}^d} e^{i(x-y) \cdot \xi} c(x, \xi) \, d\xi. \]

Next, we will consider the regularity of random sources. The following lemma shows that the random field $f$ belongs to the Hölder space and Sobolev space respectively under different values of $m$.
\begin{lemma}
	Let f be a locally isotropic Gaussian random field satisfying Assumption \ref{Ass:1}, with order -s and compact support in $M \subset \mathbb{R}^d$. \\
	$(1)$ If $s \in (d,d+2)$, then almost surely $f \in C^{0,\alpha}(M)$ for all $\alpha \in (0,\frac{s-d}{2})$. \\
	$(2)$ If $s \in \left(-\infty, d\right]$, then almost surely $f \in H^{-\frac{d-s}{2}-\epsilon}(M)$ for all sufficiently small $\epsilon>0$.
	\label{f feature}
\end{lemma}
\begin{proof}
	The proof of case \textbf{(1)} is already explained in Lemma 2.6 of the reference \cite{li2021inverse}, so we will not repeat it here. Here, we analyze case \textbf{(2)}, i.e., when $s\leq d$, the random source $f$ is relatively coarse and should be interpreted as a distribution. In the literature \cite{li2021inverse}, it is pointed out that if the strength function $\hat{\mu} \in C^{0,\alpha}(M)$ of the micro-locally isotropic Gaussian random field $\hat{f}$, then $\hat{f} \in W^{-\frac{d-s}{2}-\frac{\epsilon}{2},\hat{p}}(M)$, for sufficiently small $\epsilon > 0$ and $\hat{p} > 1$.
	
	For a random field $f$ satisfying Assumption $\ref{Ass:1}$, the strength function $\mu$ will satisfy
	\[ \sqrt{\mu} \in W_{0}^{m,4}(M) \subset W_{0}^{\frac{d-s}{2}+\frac{\epsilon}{2},q'}(M). \]
	for any $m > \frac{d}{4} + 2n - 1$ and $q' \in (2, \infty)$. Note that $f$ and $\sqrt{\mu}\hat{f}$ have the same regularity, and according to Lemma 2 in \cite{lassas2008inverse}, we have
	\[
	\left\| \sqrt{\mu} \hat{f} \right\|_{W^{- \frac{d - s}{2} - \frac{\epsilon}{2},p'}(M)} \lesssim \left\| \sqrt{\mu} \right\|_{W^{ \frac{d - s}{2} + \frac{\epsilon}{2}, q'}(M)} \left\| \hat{f} \right\|_{W^{- \frac{d - s}{2} - \frac{\epsilon}{2}, \hat{p}}(M)}
	.\]
	where $q'=\frac{2\hat{p}}{\hat{p}-1}\in (2,\infty)$ and $p'=\frac{2\hat{p}}{\hat{p}+1}\in (1,2)$, such that $\frac{1}{p'}+\frac{1}{q'}=1$.
	
	Take any $p' \in (\frac{1}{\frac{1}{2} + \frac{\epsilon}{2d}}, 2)$, and from Theorem 7.63 in the literature \cite{adams2003sobolev}, we obtain
	\[
	H_0^{\frac{d - s}{2} + \epsilon}(M) \subset W_0^{\frac{d - s}{2} + \frac{\epsilon}{2}, q'}(M).
	\]
	and
	\[
	f \in W^{-\frac{d - s}{2} - \frac{\epsilon}{2}, p'}(M) \subset H^{-\frac{d - s}{2} - \epsilon}(M).
	\]
\end{proof}

\subsection{Direct problem}
In this subsection, we will discuss the well-posedness and stability of solutions to the one-dimensional Helmholtz equation without attenuation.
\subsubsection{Fundamental solution}
When studying the one-dimensional Helmholtz equation without attenuation, analyzing the local regularity of the fundamental solution in Sobolev space is important for understanding the smoothness, existence, and properties of the solution. These analyses are particularly important for ensuring the validity of the solution and the effectiveness of the analytical methods. Specifically, the local regularity analysis in Sobolev space is directly related to the stability of numerical solutions and theoretical analyses.

According to \cite{keller1976numerical}, it is easy to see that there is a unique solution for equation (\ref{helmholtz non attenuation}) and equation (\ref{boundary condition}). Let \[ G_{\nu}(x,y)=\frac{i}{2\nu}e^{i\nu |x-y|}. \]
this solution constitutes the fundamental solution of the one-dimensional Helmholtz equation under outward radiation boundary conditions.

The following lemma will illustrate the local regularity satisfied by the fundamental solution $G_{\nu}(x,\cdot)$.

\begin{lemma}
	For any fixed $x\in \mathbb{R}$ and $1<p<\infty$, the fundamental solution $G_{\nu}(x,\cdot)\in W_{loc}^{1,p}(\mathbb{R})$.
	\label{local uniform}
\end{lemma}	
\begin{proof}
	Perform a Fourier transform on the variable $y$ (fixing $x$) for the fundamental solution $G_{\nu}(x,y)=\frac{i}{2\nu}e^{i\nu |x-y|}$
	\[
	\mathcal{F}[G_{\nu}(x, \cdot)](\xi) = \frac{i}{2\nu} \int_{-\infty}^{\infty} e^{i\nu|x - y|} e^{-i \xi y} \, dy.
	\]
	Let $z = y - x$, the right-hand side of the above equation becomes
	$\frac{i}{2\nu} e^{-i \xi x} \int_{-\infty}^{\infty} e^{i\nu|z|} e^{-i \xi z} \, dz$. Define the integral
	\[ I =\int_{-\infty}^{\infty} e^{i\nu|z|} e^{-i \xi z} \, dz = \int_{-\infty}^{0} e^{-i(\nu +\xi )z} \, dz + \int_{0}^{\infty} e^{i(\nu -\xi )z} \, dz .\]
	Calculating each integral separately, we have
	\[ \int_{0}^{\infty} e^{i(\nu -\xi )z} \, dz =\lim_{L \to \infty} \frac{e^{i(\nu - \xi)L} - 1}{i(\nu - \xi)}.
	\]
	When $L\rightarrow \infty$, $e^{i(\nu - \xi)L}$ oscillates on the unit circle and has no limit. Similarly, the integral on the negative half-axis also diverges. Therefore, we use the principal value integral with symmetric truncation
	\begin{align*}
		\text{PV} I &= \lim_{L \to \infty} \left( \int_{-L}^{0} e^{-i(\nu + \xi)z} \, dz + \int_{0}^{L} e^{i(\nu - \xi)z} \, dz \right)\\
		&=\lim_{L \to \infty} \left[ \frac{e^{i(\nu + \xi)L} - 1}{i(\nu + \xi)} + \frac{e^{i(\nu - \xi)L} - 1}{i(\nu - \xi)} \right]\\
		&= \frac{-1}{i} \left( \frac{1}{\nu + \xi} + \frac{1}{\nu - \xi} \right) = \frac{2i\nu }{\nu^2 - \xi^2}.
	\end{align*}
	i.e., $I = \frac{2i\nu }{\nu^2 - \xi^2}$, which ultimately yields
	\[ \mathcal{F}[G_{\nu}(x, \cdot)](\xi) = \frac{e^{-i\xi x}}{\xi^2 - \nu^2}. \]
	The $W_{loc}^{1,p}(\mathbb{R})$ norm can be expressed as
	\[
	\left\| G_{\nu} \right\|_{W^{1,p}(\mathcal{B}_R)} \leq C \left( \left\| G_{\nu} \right\|_{L^p(\mathcal{B}_R)} + \left\| \partial_y G_{\nu} \right\|_{L^p(\mathcal{B}_R)} \right).
	\]
	By the Hausdorff-Young inequality,
	$\left\| G_{\nu} \right\|_{L^p} \leq \left\| \mathcal{F}[G_{\nu}] \right\|_{L^{p'}}$ $\left( \frac{1}{p} + \frac{1}{p'} = 1 \right)$,
	the derivative term has $\left\| \partial_y G_{\nu} \right\|_{L^p} \leq \left\| i \xi \mathcal{F}[G_{\nu}] \right\|_{L^{p'}}$, so
	\[
	\left\| G_{\nu} \right\|_{W^{1,p}} \lesssim \left\| (1 + |\xi|) \mathcal{F}[G_{\nu}] \right\|_{L^{p'}}.
	\]
	Furthermore,
	\[
	|(1 + |\xi|) \mathcal{F}[G_{\nu}]| \leq \frac{1 + |\xi|}{|\xi^2 - \nu^2|} \leq \frac{1 + |\xi|}{|\xi|^2 - |\nu|^2}.
	\]
	When $|\xi|\geq 2|\nu|$, we have $\frac{1 + |\xi|}{|\xi|^2 - |\nu|^2} \leq \frac{2|\xi|}{|\xi^2/2|} = \frac{4}{|\xi|}$, $\int_{|\xi| \geq 2|\nu|} \left( \frac{4}{|\xi|} \right)^{p'} d\xi \leq 8 \int_{|\xi| \geq 2|\nu|} \xi^{-p'} d\xi < \infty \quad \text{if} \quad p' > 1.$
	When $|\xi|\leq 2|\nu|$, we have $\frac{1 + |\xi|}{|\xi^2 - \nu^2|}\leq \frac{1 + 2|\nu|}{\text{dist}(\xi^2, \nu^2)}$. Since the denominator does not disappear on the compact set, $\int_{|\xi| \leq 2|\nu|} \left( \frac{1 + 2|\nu|}{\text{dist}(\xi^2, \nu^2)} \right)^{p'} d\xi < \infty$. $L^{p'}$'s integrability can be proven. Therefore, $\left\| G_{\nu} \right\|_{W^{1,p}(\mathcal{B}_R)} \lesssim \left\| (1 + |\xi|) \mathcal{F}[G_{\nu}] \right\|_{L^{p'}}<\infty.$
\end{proof}	

Using the fundamental solution $G_{\nu}$, we now discuss its well-posedness. Firstly, we define the volume potential
\[ (V_{\nu} f)(x):=-\int_{\mathbb{R}} G_{\nu}(x,y)f(y) \,dy. \]
It describes a potential field caused by a source term (such as mass or charge distribution) with the following properties.
\begin{lemma}
	Let $I$ and $O$ be two bounded intervals on $\mathbb{R}$, and let the operator $V_{\nu}:H^{-\beta}(I)\rightarrow H^{\beta}(O)$ be bounded, where $\beta\in (0,1]$.
	\label{operator H}
\end{lemma}
The proof of the above lemma is given in Lemma 2.3 of \cite{li2020inverse}, and it is omitted here. It provides information about the boundedness of the operator $V_{\nu}$ between Sobolev spaces. In short, the boundedness of the solution operator propagates the Sobolev regularity from $f$ to $u$.

The following theorem proves the well-posedness of the direct problem $(\ref{helmholtz non attenuation})-(\ref{boundary condition})$ in the sense of distributions.
\begin{theorem}
	Suppose that $f$ satisfies Assumption $\ref{Ass:1}$ and $s\in(-\frac{2}{n},1]$, then the stochastic Helmholtz equation $(\ref{helmholtz non attenuation})-(\ref{boundary condition})$ has a unique solution in the sense of distributions which is given as follows:
	\[ u(x;\nu) = -\int_{M} G_{\nu}(x, y)f(y) \,dy.  \] 
	It almost surely holds that $u\in W_{loc}^{\gamma,n}(\mathbb{R})$ for any $n>1$, and \[ \frac{1-s}{2}<\gamma<\frac{1}{2}+\frac{1}{n}. \]
	\label{helm direct}
\end{theorem}
\begin{proof}
	We only need to prove the existence of the solution, because uniqueness can be directly obtained from the deterministic case. Define $I = (-i, i)$. By Lemma \ref{local uniform}, we have $G_{\nu}(x,\cdot) \in W^{1,m}(I) \hookrightarrow W^{\gamma,n}(I)$ for some $m > 1$ such that $\frac{1}{m} - \frac{1}{n} < 1 - \gamma$. For $f$ in Assumption \ref{Ass:1}, by Lemma \ref{f feature}, for any $\epsilon>0$ and $p>1$, $f\in W^{\frac{s-d}{2}-\epsilon,p}(M)$. For any $x\in \mathbb{R}$, define the volume potential
	\[ u_{*}(x;\nu) = -\int_{M} G_{\nu}(x,y)f(y) \,dy = -\int_{\mathbb{R}}G_{\nu}(x,y)f(y) \,dy. \]
	Firstly, we need to prove that $u_{*}$ is a solution in the sense of distributions. In fact, for any $k\in W_{loc}^{\gamma,n}(\mathbb{R})$,
	\begin{align*}
		\langle \Delta u_* + \nu^2 u_*, k \rangle &= -\langle \nabla u_*, \nabla k \rangle + \nu^2 \langle u_*, k \rangle \\
		&= \int_{\mathbb{R}} \nabla_x \left[ \int_{M} G_{\nu}(x, y) f(y) \, dy \right] \nabla k(x) \, dx - \nu^2 \int_{\mathbb{R}} \left[ \int_{M} G_{\nu}(x, y) f(y) \, dy \right] k(x) \, dx\\
		&= - \int_{M} \int_{\mathbb{R}} \Delta_x G_{\nu}(x, y) k(x) f(y) \, dx \, dy - \nu^2 \int_{\mathbb{R}} \int_{M} G_{\nu}(x, y) f(y) \, dy \, k(x) \, dx\\
		&= \int_{M} \int_{\mathbb{R}} \left( \nu^2 G_{\nu}(x, y) + \delta(x - y) \right) k(x) f(y) \, dx \, dy\\
		&= \langle f, k \rangle.
	\end{align*}
	To prove that $u\in W_{loc}^{\gamma,n}(\mathbb{R})$, it suffices to prove that $\eta u_* \in W_{loc}^{\gamma,n}(\mathbb{R})$, where $\eta\in C_{0}^{\infty}$ is a bounded function with compact support in $U\subset \mathbb{R}$. Define a weighted potential
	\[ (\hat{V}_{\nu} f)(x):=-\eta(x) \int_{\mathbb{R}} G_{\nu}(x,y)f(y) \,dy,  \quad x\in U.\]
	By Lemma \ref{operator H}, the operator $\hat{V}_{\nu}:H^{-\beta}(M)\rightarrow H^{\beta}(U)$ is bounded for any $\beta\in (0,1]$. For the parameters $\gamma$ and $n$ assumed in the theorem, choose $\beta=1$ and $\frac{1}{m}+\frac{1}{n}=1$. According to the Kondrachov embedding theorem, we have
	\[ W^{-\gamma,m}(M) \hookrightarrow H^{\beta}(M),\quad H^{-\beta}(U)\hookrightarrow W^{\gamma,n}(U).\]\\
	is continuous, ultimately yielding that $\hat{V}_{\nu}:W^{-\gamma,m}(M)\rightarrow W^{\gamma,n}(U)$ is bounded, which implies that $\eta u_* = \hat{V}_{\nu}\in W^{\gamma,n}. $
\end{proof}

\subsection{Finite-jump Lévy processes}
The Lévy process \cite{applebaum2009levy} is a class of stochastic processes with independent and stationary increments, characterized by paths that can be described by continuous and jump components. A finite-jump Lévy process is a subclass of Lévy processes, where the jump component occurs only a finite number of times in any finite time interval. Namely, there exists a Lévy measure $\tau$ such that $\tau(\mathbb{R})<\infty$. According to the Lévy-Itô decomposition, the decomposition of the finite jump Lévy process $L_t$ is as follows
\[ L_t = bt + \sigma W_t + \sum_{0<s\leq t} J_s.\]
where $b$ is the drift coefficient, $\sigma W_t$ is the diffusion term of standard Brownian motion, $\sigma$ is the volatility, and $\sum_{0<s\leq t} J_s$ represents the finite jump process, where $J_s$ is the jump amplitude and the number of jump times $s$ is finite within a finite time interval.\\
Characteristics of the finite-jump Lévy process
\begin{itemize}[left=5pt]
	\item The jump amplitude $J_s$ usually follows a given distribution (such as normal or exponential distribution);
	\item The interval between jump times follows a Poisson distribution, and the number of jumps is finite.
\end{itemize}
These properties ensure that the path of $L_t$ has the property of being left-extreme and right-connected.

The form of the stochastic wave equation driven by a finite-jump Lévy process is as follows
\begin{equation}
	u_{tt} - \Delta u = f(x)h(t) + g(x)\dot{L}_t, \quad (x,t) \in D \times [0,T].
	\label{levy}
\end{equation}
where the boundary conditions $u(x,t)=0$, $(x,t) \in \partial D \times [0,T]$, the initial conditions $u(x,0)=u_t(x,0)=0$, $x \in \bar{D}$. $D \subset \mathbb{R}^d$ is a bounded domain with a Lipschitz boundary $\partial D$.

\begin{definition}
	A stochastic process $u(x,t)$ taking values in $L^2(D)$ is called a mild solution of equation $(\ref{levy})$ and can be expressed as
	\begin{equation}
		u(x,t) = \int_0^t K(x,t-\tau)f(x)h(\tau) \, d\tau + \int_0^t K(x,t-\tau)g(x) \, dL_\tau. 
		\label{mild}
	\end{equation}
\end{definition}
The kernel function $K(x,t-\tau) = \sin\left( (t-\tau)\sqrt{-\Delta} \right) (-\Delta)^{-1/2}$ and $dL_\tau$ is the increment of the $L\acute{e}vy$ jump process.
\[ dL_\tau = b d\tau + \sigma dW_\tau + \sum_{s \in J_\tau} J_s \delta_s. \]
where $\delta_s$ is the Dirac measure at the jump point. Therefore, the mild solution can be decomposed into the following parts
\[ u(x,t) = u_{\text{det}}(x,t) + u_{\text{diff}}(x,t) + u_{\text{jump}}(x,t). \]
1. Determined part
\begin{equation}
	u_{\text{det}}(x,t) = \int_0^t K(x,t-\tau)f(x)h(\tau) \, d\tau + \int_0^t K(x,t-\tau)g(x)b \, d\tau.
\end{equation}
2. Diffusion part(Brownian motion driven)
\begin{equation}
	u_{\text{diff}}(x,t) = \int_0^t K(x,t-\tau)g(x)\sigma \, dW_\tau.
\end{equation}
3. Jumping part
\begin{equation}
	u_{\text{jump}}(x,t) = \int_0^t K(x,t-\tau) g(x) \left( \sum_{s \in J_{\tau}} J_s \delta_s \right).
	\label{jump}
\end{equation}

\begin{remark}
	The operator $-\Delta$ with homogeneous Dirichlet boundary conditions has eigenvalues and eigenvectors $\{\lambda_{k},\varphi_{k}\}_{k=1}^{\infty}$, where the eigenvalues satisfy: as $k \rightarrow \infty$, $\lambda_{k}\rightarrow \infty$, $0 < \lambda_1 < \lambda_2 < \cdots < \lambda_k < \cdots$, and the eigenfunctions $\{\varphi_k\}_{k=1}^{\infty}$ form a complete orthogonal basis in the $L^2(D)$ space. For any function $s(x)$ in the $L^2(D)$ space, it can be written as
	\[ s(x)=\sum_{k=1}^{\infty}s_k\varphi_{k}(x), \quad s_k=(s,\varphi_k)_{L^2(D)}=\int_{D}s(x)\varphi_{k}(x)\, dx. \]
	Therefore, if $u \in L^2(D)$ is a mild solution to equation $(\ref{levy})$, then
	\[ u(\cdot, t) = \sum_{k=1}^{\infty} u_k(t) \varphi_k. \]
	where \[ u_k(t) = \left( u(\cdot, t), \varphi_k \right)_{L^2(D)} 
	= f_k \int_0^t \frac{\sin\left((t - \tau)\sqrt{\lambda_k}\right)}{\sqrt{\lambda_k}} h(\tau) d\tau + g_k \int_0^t \frac{\sin\left((t - \tau)\sqrt{\lambda_k}\right)}{\sqrt{\lambda_k}} dL_\tau.
	\]
	$f_k = (f, \varphi_k)_{L^2(D)}, g_k = (g, \varphi_k)_{L^2(D)}$, and $u_k(t)$ satisfies the following stochastic differential equation
	\begin{equation*}
		\left\{
		\begin{aligned}
			&u_k''(t) + \lambda_k u_k(t) = f_k h(t) + g_k \dot{L}_t, \quad t \in (0, T), \\
			&u_k(0) = u_k'(0) = 0. 
		\end{aligned}
		\right.
	\end{equation*}
\end{remark}

\section{The inverse source problem for the one-dimensional stochastic Helmholtz equation without attenuation}
In this section, we discuss the inverse source problem for the one-dimensional stochastic Helmholtz equation without attenuation. The main focus is on implementing an efficient numerical method for point-by-point reconstruction of the strength function of the stochastic source term from a computational perspective.

According to Theorem \ref{helm direct}, we have \[ u(x) = \frac{1}{2i\nu} \int_{M} e^{i\nu |x-y|} f(y) \, dy. \]
The key to solving the inverse problem is typically to infer the micro-local strength $\mu(y)$ of the source by analyzing the variance or other statistical characteristics of the observed wavefield $u(x)$. This is typically achieved by constructing a mathematical model and combining it with measurement data for reconstruction.

\subsection{Inverse source problem}
For the source term \[ f(x) = \sqrt{\mu(x)} (-\Delta)^{-\frac{s}{4}} W_{x}^{'}. \]
We consider the case where $s = 0$, i.e., \[ f(x) = \sqrt{\mu(x)} W_{x}^{'}. \]
where $\mu(x)$ is a smooth function with compact support on $M=(0,1)$. By Lemma \ref{f feature}, there exists a sufficiently small $\epsilon>0$ such that $f \in H^{-\frac{d-s}{2}-\epsilon}(M)$. The covariance operator $C_f$ of $f$ will satisfy
\begin{align*}
	\langle \varphi,C_f \psi \rangle &= \mathbb{E} \left[ \langle f, \varphi \rangle \langle f, \psi \rangle \right]\\
	&= \mathbb{E}\left[\int \mu(y) W'_y \varphi(y) \, dy \cdot \int \mu(z) W'_z \psi(z) \, dz\right]\\
	&=\int_{0}^{1} \mu(y)\varphi(y) \psi(y) \, dy\\
	&= \int_{M} \int_{M} K_f(x, y) \varphi(x) \psi(y) \, dx \, dy. 
\end{align*} 
and \[ \delta(x-y)=\frac{1}{2\pi}\int_{M}e^{i(x-y)\xi}\mu(x) \, d\xi. \] Therefore, \[ K_f(x, y) = \mu(y) \delta(x-y) = \frac{1}{2\pi}\int_{M}e^{i(x-y)\xi}\mu(x) \, d\xi. \]
So we can see that $c(x, \xi) = \mu(x)$, and $f$ satisfies Assumption \ref{Ass:1}.
From this, we can see that the solution to equation (\ref{helmholtz non attenuation}) can be expressed as
\begin{equation}
	u(x) = \frac{1}{2i\nu}\int_{M}e^{i\nu|x-y|}\sqrt{\mu(y)} \, dW(y),  \quad  x\in \mathbb{R}.
	\label{reconstruction}
\end{equation}

The random component of the wave field $u(x)$ often complicates analytical and numerical solutions. In numerical simulations, the random component often introduces uncertainty and instability. By using the Itô formula, the random component of the wave field can be separated from its statistical properties. This allows the extraction of statistical quantities such as the mean and variance of the wave field, facilitating statistical analysis of the wave field.

Apply the Itô formula $\mathbb{E}|\int_{0}^{1}f(y) \, dW(y)|^2 = \int_{0}^{1}|f(y)|^2 \, dy$ to calculate the expected value of the second moment of $u(x)$
\begin{equation}
	\mathbb{E}|u(x)|^2 = \frac{1}{4|\nu|^2}\int_{M}\mu(y) \, dy. 
	\label{ill}
\end{equation}

By(\ref{ill}), we can see that in the absence of attenuation, the variance of the wavefield can only provide the average value of the random source strength, but cannot accurately reconstruct the strength point by point. To determine the strength point by point, more detailed statistical information or further regularization processing is required.

\subsubsection{Numerical experimental design}
Data discretization processing: $x$ is used to calculate the position point of the corrected data, and the discrete points $\{x_m\}_{m=0}^{M-1}$ take values in the interval $I=[-1.2,-0.2]\cup[1.2,2.2]$, and define
$$x_0 = -1.2, \quad \Delta x = 2/M, \quad x_{m+1} = x_m + \Delta x, \quad M = 200, \quad m = 0, \dots, M-1.$$
$y$ is the grid point used for integration, i.e., the domain of the strength distribution. Since the strength is reconstructed on $[0,1]$, we define
$$ y_0 = 0, \quad \Delta y = 1/N, \quad y_{n+1} = y_n + \Delta y, \quad N = 200, \quad n = 0, \dots, N-1.$$
According to equation $(\ref{reconstruction})$, $u(x) = \frac{1}{2i\nu}\int_{M}e^{i\nu|x-y|}\sqrt{\mu(y)} \, dW(y)$, we have\[ 2i\nu u(x) = \int_{M}e^{i\nu|x-y|}\sqrt{\mu(y)} \, dW(y). \]

The characteristic function $H(x,\nu,\mu)$ can be defined by taking the real part of a complex integral. In physics and engineering, we are often only concerned with the real part of a physical quantity represented by a complex number. For example, the solutions to the wave equation are typically expressed as complex numbers, but the actual physical quantities, such as displacement and pressure, are the real parts of the complex solutions.

Since $\mathbb{E}|u(x)|^2 = \frac{1}{4|\nu|^2}\int_{M}\mu(y) \, dy$, we have
\begin{equation}
	4|\nu|^2\cdot \mathbb{E}|u(x)|^2 = \int_{M}\mu(y) \, dy.
	\label{h0}
\end{equation} 
Separating the real and imaginary parts, we get
\[  \text{Re}(2viu(x)) = \int_{M} \cos(\nu \mid x - y \mid) \cdot \sqrt{\mu(y)} \, dW(y). \]
\[ \text{Im}(2viu(x)) = \int_{M} \sin(\nu \mid x - y \mid) \cdot \sqrt{\mu(y)} \, dW(y). \]

There are two methods for defining the feature function $H(x,\nu,\mu)$.\\
1. By taking the real part and squaring it, then calculating the expectation, we can obtain a numerically stable feature.
\begin{equation}
	H^1(x, \nu, \mu) = \mathbb{E} \left| \text{Re}(2viu(x)) \right|^2 = \int_{M} \cos^2(\nu \mid x - y \mid) \cdot \mu(y) \, dy.
	\label{h1}
\end{equation}
2. By squaring the real and imaginary parts separately and then subtracting them, the influence of the imaginary part on strength recovery can be completely eliminated (the imaginary part is not always an interference component; sometimes it also provides additional information).
\begin{align}
	H^2(x, \nu, \mu) &= \mathbb{E} \left| \text{Re}(2viu(x)) \right|^2 - \mathbb{E} \left| \text{Im}(2viu(x)) \right|^2 \nonumber\\
	&= \int_{M} \cos(2\nu \mid x - y \mid) \cdot \mu(y) \, dy.
	\label{h2}
\end{align}
Comparing the right-hand sides of equations $(\ref{h0})$ and $(\ref{h1})(\ref{h2})$, the coefficients $\cos^2(\nu \mid x - y \mid)$ and $\cos(2\nu \mid x - y \mid)$ are similar to a dynamic weighting coefficient assigned in the strength reconstruction process. In order to maximize the retention of the characteristics of the real part while eliminating the influence of the imaginary part on strength recovery, making the recovery process more stable, and taking into account the forms of the two coefficients, we use discretization and summation to approximate the integral
\[ H(x, \nu, \mu) \approx \Delta y \sum_{j=1}^{N} \cos^2(2\nu \mid x - y_j \mid) \mu(y_j).
\]

Next, Tikhonov regularization is applied to the objective function, defined as follows
\[ J(\mu) = \sum_{\nu \in \text{frequencies}} \quad \sum_{x \in x_\text{points}} \left( H_{\text{obs}}(x, \nu, \mu) - H(x, \nu, \mu) \right)^2 + \alpha \| \mu \|_2^2. \]
$H_{\text{obs}}(x, \nu, \mu)$ is the observed data at frequency $\nu$ and observation point $x$, and $| \mu \|_2^2 = \sum_{j=1}^{N}\mu(y_j)^2$.\\
The observed data $H_{\text{obs}}(x, \nu, \mu)$ in the above equation is generated using different frequency data $\nu$ and the known true strength distribution, while the initial strength distribution is used for the iterative calculation of $H(x, \nu, \mu)$ in the above equation.

\subsubsection{Multi-frequency data fusion and regularization processing}
In the inverse problem of the one-dimensional Helmholtz equation without attenuation, using the second-order moment expectation of the wavefield $u(x)$ can indeed help filter out random noise from the source term, especially the random fluctuations introduced by Brownian motion increments. However, even so, measurable data still contains non-negligible noise, primarily due to uncertainties in finite sample estimates, errors introduced by model simplification, and insufficient suppression of high-frequency noise by the second-order moment expectation.

Therefore, the uniqueness and stability of the reconstructed wavefield $u(x)$ depend on how noise and data incompleteness are handled. This requires the integration of advanced numerical methods, such as regularization techniques and multi-frequency data fusion, to ensure the uniqueness and stability of the reconstruction results.

\textbf{Regularization} is a technique introduced to address the ill-posedness of inverse problems (i.e., solutions that are non-unique, unstable, or discontinuous). By incorporating a regularization term (i.e., additional constraints), stable solutions can be obtained while minimizing errors.

\textbf{Multi-frequency data fusion} utilizes observational data at different frequencies to enhance the stability and uniqueness of solutions to inverse problems. Multi-frequency data effectively overcomes the non-uniqueness issues caused by the periodicity of waves.

Next, we systematically discuss the ill-posedness of the inverse problem, the necessity of multi-frequency data fusion, and the stability and convergence of regularization methods. We use a combination of continuous and discrete analysis, with the continuous model revealing the essential ill-posedness and the discrete model corresponding to actual calculations.

\begin{definition}[Direct Problem]\label{def:forward_problem}
	Let $\mu \in L^2([0,1])$ be an unknown function and $\nu > 0$ be a frequency. Define the integral operator $T_\nu : L^2([0,1]) \to L^2([0,1])$ as
	\[(T_\nu \mu)(x) = \int_0^1 K_\nu(x,y)\mu(y)dy, \quad K_\nu(x,y)=\cos^2(2\nu|x-y|).\]
\end{definition}

For the inverse problem, given a finite set of frequencies $\{\nu_k\}_{k=1}^N \subset \mathbb{R}^+$ and observed data $H^{\text{obs}}(\cdot, \nu_k) \in L^2([0,1])$, find $\mu \in L^2([0,1])$ such that
\[ T_{\nu_{k}} \mu = H^{\text{obs}}(\cdot, \nu_k), \quad k = 1, \dots, N. \]

In practice, observation points $\{x_i\}_{i=1}^M \subset I$ for discrete data $H^{\text{obs}}(x_i, \nu_j)$.The discretization grid is defined as \( y_j = j \Delta y \), where \( \Delta y = 1/N \). The unknown function \( \mu \) is represented by the vector \( \mu = (\mu_1, \ldots, \mu_N)^T \in \mathbb{R}^N \) with \( \mu_j = \mu(y_j) \). The matrix \( A^{(k)} \in \mathbb{R}^{M \times N} \) for frequency \( \nu_k \) is defined by

\[
A_{ij}^{(k)} = \Delta y \cos^2(2\nu_k |x_i - y_j|), \quad i = 1, \ldots, M, \, j = 1, \ldots, N.
\]

The global matrix \( A \in \mathbb{R}^{MK \times N} \) is constructed by column-wise concatenation of the submatrices
\[
A = \begin{bmatrix} A^{(1)} \\ \vdots \\ A^{(K)} \end{bmatrix}
\]
and the observed data vector \( H^{\text{obs}} \in \mathbb{R}^{MK} \) is the concatenation of all \( H^{\text{obs}}(x_i, \nu_k) \) values. The objective function is given as follows.

\[
J(\mu) = \|H^{\text{obs}} - A\mu\|_2^2 + \alpha \|\mu\|_2^2.
\]
where \( \|\mu\|_2^2 = \sum_{j=1}^N \mu_j^2 \) is the squared \( L^2 \)-norm of \( \mu \). The inverse problem consists of minimizing \( J(\mu) \).

Next we will characterize the ill-posedness of the inverse problem  at a single frequency.
\begin{lemma}[Compactness]\label{lem:compactness}
	For any $\nu > 0$, the operator $T_\nu: L^2([0,1]) \to L^2([0,1])$ is compact.
\end{lemma}

\begin{proof}
	The kernel $K_\nu(x,y) = \cos^2(2\nu|x-y|)$ can be expressed as
	\[
	K_\nu(x,y) = \frac{1}{2} + \frac{1}{2}\cos(4\nu|x-y|).
	\]
	Since the cosine function is bounded and continuous, $K_\nu$ is a bounded continuous kernel. This implies that $T_\nu$ is a Hilbert-Schmidt operator, and hence compact (cf.\cite{reed1980methods}).
\end{proof}

\begin{lemma}[Self-adjointness]\label{lem:selfadjoint}
	$T_\nu$ is self-adjoint. i.e. $\langle T_\nu \mu, \phi \rangle = \langle \mu, T_\nu \phi \rangle$ for all $\mu, \phi \in L^2([0,1])$.
\end{lemma}

\begin{proof}
	The kernel satisfies $K_\nu(x,y) = K_\nu(y,x)$ due to the symmetry of $|x-y|$. Therefore, $T_\nu$ is self-adjoint.
\end{proof}

\begin{theorem}[Ill-posedness]\label{thm:illposedness}
	For a fixed $\nu > 0$, the inverse problem $T_\nu \mu = H$ is ill-posed.\\
	$(1)$ The singular values $\{\sigma_n(\nu)\}$ satisfy $\sigma_n(\nu) \to 0$ as $n \to \infty$, so the inverse operator is unbounded.\\
	$(2)$ The null space $\mathcal{N}(T_\nu)$ is non-trivial for certain $\nu$.\\
	$(3)$ The solution does not depend continuously on the data.
\end{theorem}

\begin{proof}
	\textbf{(1)} By Lemma \ref{lem:compactness}, $T_\nu$ is compact. The spectral theorem for compact self-adjoint operators (Lemma \ref{lem:selfadjoint}) guarantees that the singular values $\sigma_n(\nu)$ accumulate at zero(cf.\cite{kress1999linear}, Chapter 15). 
	Indeed, since $T_{\nu}$ be a compact self-adjoint operator, according to the theory of Singular Value Decomposition (SVD), there exists a sequence of non-negative real numbers \( \{\sigma_n(\nu)\}_{n=1}^\infty \) , and orthonormal bases \( \{\phi_n\}_{n=1}^\infty \), \( \{\psi_n\}_{n=1}^\infty \), such that
	\[
	T_\nu \mu = \sum_{n=1}^\infty \sigma_n(\nu) \langle \mu, \phi_n \rangle \psi_n, \quad \forall \mu \in L^2([0,1]).
	\]
	where \( \phi_n = \psi_n \) (\(T_\nu\) self-adjoint), and \( \sigma_n(\nu) \) are the absolute values of the eigenvalues of \( T_\nu \).
	\[
	T_{\nu}^{-1} \psi_n = \frac{1}{\sigma_n(\nu)} \phi_n.
	\]
	
	\[
	\text{As } \sigma_n(\nu) \to 0, \quad \|T_{\nu}^{-1}\| = \sup_n \frac{1}{\sigma_n(\nu)} \to \infty.
	\]

	\textbf{(2)} We construct a non-zero $\mu$ such that $T_\nu \mu = 0$. For example, take $\mu(y) = \cos(8\nu y)$. Then
	\[
	T_\nu \mu(x) = \int_0^1 \cos^2(2\nu|x-y|)\cos(8\nu y)\,dy.
	\]
	Using trigonometric identities, we have
	\begin{align*}
		\cos^2(2\nu |x-y|) \cos(8\nu y) &= \frac{1}{2} \left[1 + \cos(4\nu |x-y|)\right] \cos(8\nu y) \\
		&= \frac{1}{2} \cos(8\nu y) + \frac{1}{4} \left[ \cos(4\nu |x-y| + 8\nu y) + \cos(4\nu |x-y| - 8\nu y) \right].
	\end{align*}
	For specific $\nu$ (e.g., $\nu = \pi/2$), this integral vanishes when integrated against $\cos(8\nu y)$, demonstrating non-trivial null space. 
	
	\textbf{(3)} Follows from \textbf{(1)} since the singular values $\sigma_n(\nu) \to 0$ imply that the pseudo-inverse $T_\nu^\dagger$ is unbounded, so small data perturbations can cause large solution changes (cf.\cite{engl1996regularization}, Theorem 2.14).  
\end{proof}

\begin{remark}
	Suppose the observational data \( H \) is perturbed slightly by \( \delta H \), then the corresponding perturbation of the solution is \( \delta \mu \).
	\[
	T_\nu(\mu + \delta \mu) = H + \delta H \quad \Longrightarrow \quad \delta \mu = T_\nu^\dagger \delta H.
	\]
	Since \( \|T_\nu^\dagger\| = +\infty \), even if \( \|\delta H\| \) is small, \( \|\delta \mu\| \) may still become arbitrarily large.
	
	Take the perturbation \( \delta H = \epsilon \psi_{n_k} \), where \( \psi_{n_k} \) corresponds to a singular value \( \sigma_{n_k}(\nu) \), then
	
	\[
	\delta \mu = \frac{\epsilon}{\sigma_{n_k}(\nu)} \phi_{n_k}
	\quad \Longrightarrow \quad
	\|\delta \mu\| = \left| \frac{\epsilon}{\sigma_{n_k}(\nu)} \right| \to +\infty
	\quad \text{as } \sigma_{n_k}(\nu) \to 0.
	\]
\end{remark}

\begin{remark}\label{rem:discrete_illposedness}
	In the discrete setting, for a matrix \( A^{(k)} \in \mathbb{R}^{M \times N} \), its Singular Value Decomposition (SVD) is given by
	\[
	A^{(k)} = U \Sigma V^T, \quad \Sigma = \text{diag}(\sigma_1, \sigma_2, \ldots, \sigma_r), \quad \sigma_1 \geq \sigma_2 \geq \cdots \geq \sigma_r > 0.
	\]
	where \( r = \min(M, N) \), and \( \sigma_i \) are the singular values. the matrix $A^{(k)}$ has exponentially decaying singular values, i.e., $\sigma_j \sim e^{-cj}, \, c > 0$, leading to large condition numbers \( \kappa(A^{(k)}) = \dfrac{\sigma_{\max}(A^{(k)})}{\sigma_{\min}(A^{(k)})} \) and numerical instability.
	In fact, when solving the linear system \( A^{(k)} \mu = H \), the relative error in the solution may be amplified
	\[
	\frac{\|\delta \mu\|}{\|\mu\|} \leq \kappa(A^{(k)}) \cdot \frac{\|\delta H\|}{\|H\|}.
	\]
	Even if the data perturbation \( \delta H \) is small, a large condition number can cause the solution error \( \delta \mu \) to grow explosively.
\end{remark}

The following content will elaborate on the recoverability of multi-frequency fusion, namely, a frequency set with sufficient density can guarantee the triviality of the zero-space.
\begin{theorem} (Uniqueness with dense frequencies)\label{Uniqueness}
	Let the frequency set $\{\nu_k\}_{k=1}^{\infty} \subset \mathbb{R}^+$ satisfy that $\{4\nu_k\}$ is dense in $[0,\infty)$. If $T_{\nu_k} \mu = 0$ for all $k$, then $\mu = 0$ in $L^2([0,1])$.
\end{theorem}

\begin{proof}
	Assume $T_{\nu_k} \mu = 0$ for all $k$, i.e.,
	\[
	\int_0^1 \cos^2(2\nu_k|x - y|) \mu(y) \, dy = 0, \quad \forall x \in [0,1], \, \forall k.
	\]
	By $K_{\nu}(x, y) = \frac{1}{2} + \frac{1}{2} \cos(4\nu(x - y)) \,(\text{since} \, \cos(4\nu |x - y|) = \cos(4\nu(x - y)))$, which is equivalent to
	\[
	\frac{1}{2} \int_0^1 \mu(y) \, dy + \frac{1}{2} \int_0^1 \cos(4\nu_k(x - y)) \mu(y) \, dy = 0, \quad \forall x, \, \forall k.
	\]
	Fix $x$, and let $k \to \infty$. The oscillatory integral $\int_0^1 \cos(4\nu_k(x - y)) \mu(y) \, dy \to 0$ (Riemann-Lebesgue Lemma), therefore
	\[
	\frac{1}{2} \int_0^1 \mu(y) \, dy = 0 \implies \int_0^1 \mu(y) \, dy = 0.
	\]
	Consequently,
	\[
	\int_0^1 \cos(4\nu_k(x - y)) \mu(y) \, dy = 0, \quad \forall x, \, \forall k.
	\]
	Fix $\nu_k$, the formula holds for all $x$. Let $x = 0$,
	\[
	\int_0^1 \cos(4\nu_k y) \mu(y) \, dy = 0, \quad \forall k.
	\]
	
	Define 
	\[
	\hat{\mu}_c(\omega) = \int_0^1 \cos(\omega y) \mu(y) \, dy.
	\]
	Then $\hat{\mu}_c(4\nu_k) = 0$ for all $k$. Since $\{4\nu_k\}$ is dense in $[0,\infty)$ and $\hat{\mu}_c$ is continuous, $\hat{\mu}_c(\omega) = 0$ for all $\omega \geq 0$. The system $\{\cos(n\pi y)\}_{n=0}^\infty$ is complete in $L^2([0,1])$, so $\mu = 0$(cf.\cite{folland1999real}).
\end{proof}

\begin{corollary}
	In the discrete setting with sufficiently large $K$ and dense frequency sampling, the null space $\mathcal{N}(A)$ is trivial.
	The null space of a matrix \( A \) is defined as
	\[
	\mathcal{N}(A) = \left\{ \mu \in \mathbb{R}^N \,\middle|\, A\mu = 0 \right\}.
	\]
	If \( \mathcal{N}(A) \) is trivial (i.e., it only contains the zero vector \( \mu = 0 \)), then the solution to the inverse problem is unique.
\end{corollary}

Subsequently, we will analyze how Tikhonov regularization makes the inverse problem well-posed.
\begin{theorem}[Existence and uniqueness]
	\label{thm:existence}
	For any $\alpha > 0$ and $H^{\text{obs}} \in L^2([0,1])$, the continuous objective functional is expressed as
	\[
	J(\mu) = \sum_{k=1}^K \| H^{\text{obs}}(\cdot, \nu_k) - T_{\nu_k} \mu \|_{L^2}^2 + \alpha \| \mu \|_{L^2}^2.
	\]
	which has a unique minimizer $\mu_\alpha \in L^2([0,1])$. In the discrete case, for any $\mathbf{H}^{\text{obs}} \in \mathbb{R}^{MK}$, there exists a unique $\mu_\alpha \in \mathbb{R}^N$ minimizing $J(\mu)$.
\end{theorem}

\begin{proof}
	\textit{Continuous case:} The functional $J$ is coercive and strictly convex since $\alpha > 0$. Specifically,
	\[
	J(\mu) \geq \alpha \| \mu \|_{L^2}^2 \to \infty \quad \text{as} \quad \| \mu \|_{L^2} \to \infty.
	\]
	and the Hessian operator $H_\mu = 2 \sum_{k=1}^K T_{\nu_k}^* T_{\nu_k} + 2\alpha I$ is strictly positive definite, where $T_{\nu_k}^*$ is the adjoint operator of $T_{\nu_k}$. Indeed, For any non-zero $\mu \in L^2([0,1])$,
	\[
	\langle H_\mu \mu, \mu \rangle = 2\sum_{k=1}^K \| T_{\nu_k} \mu \|_{L^2}^2 + 2\alpha \| \mu \|_{L^2}^2 > 0.
	\]
	Since $\alpha> 0$ and $\| \mu \|_{L^2}^2>0$, $H_\mu$ is strictly positive definite. Hence $J$ attains a unique minimum in $L^2([0,1])$(cf.\cite{ekeland1999convex}).
	
	\textit{Discrete case:} For any non-zero $\mu \in \mathbb{R}^N$,
	\[
	\mu^T (A^T A + \alpha I) \mu = \| A\mu \|_2^2 + \alpha \| \mu \|_2^2 > 0.
	\]
	Since $\alpha > 0$ and $\| \mu \|_2^2 > 0$, the Hessian matrix $\nabla^2 J(\mu) = 2A^T A + 2\alpha I$ is positive definite, so $J$ is strictly convex. Therefore, a unique minimizer exists.
\end{proof}

Next theorem shows that the change of the solution $\|\mu_{\alpha,1} - \mu_{\alpha,2}\|$ is amplified at most by a factor of $\frac{1}{\sqrt{\alpha}}$ with $\alpha > 0$.
\begin{theorem}[Stability]
	\label{thm:stability}
	The regularized solution $\mu_\alpha$ depends continuously on the data. If $\| H^{\text{obs}}_1 - H^{\text{obs}}_2 \| < \delta$ (continuous: $L^2$ norm; discrete: $\ell^2$ norm), then
	\[
	\| \mu_{\alpha,1} - \mu_{\alpha,2} \| \leq C(\alpha) \delta.
	\]
	where $C(\alpha) = \frac{1}{\sqrt{\alpha}}$.
\end{theorem}

\begin{proof}
	\textit{Continuous case:} Let $J_i(\mu) = \sum_{k=1}^K \| H_i^{\text{obs}} - T_{\nu_k} \mu \|_{L^2}^2 + \alpha \| \mu \|_{L^2}^2$ for $i=1,2$. By the minimizing property,
	\[
	J_1(\mu_{\alpha,1}) \leq J_1(\mu_{\alpha,2}), \quad J_2(\mu_{\alpha,2}) \leq J_2(\mu_{\alpha,1}).
	\]
	Adding these inequalities, we get
	\begin{align*}
		&\sum_k \| H_1^{\text{obs}} - T_{\nu_k} \mu_{\alpha,1} \|^2 + \alpha \| \mu_{\alpha,1} \|^2 + \sum_k \| H_2^{\text{obs}} - T_{\nu_k} \mu_{\alpha,2} \|^2 + \alpha \| \mu_{\alpha,2} \|^2 \\
		&\leq \sum_k \| H_1^{\text{obs}} - T_{\nu_k} \mu_{\alpha,2} \|^2 + \alpha \| \mu_{\alpha,2} \|^2 + \sum_k \| H_2^{\text{obs}} - T_{\nu_k} \mu_{\alpha,1} \|^2 + \alpha \| \mu_{\alpha,1} \|^2.
	\end{align*}
	By simplifying, we have
	\[
	\sum_k \left[ \| H_1^{\text{obs}} - T_{\nu_k} \mu_{\alpha,1} \|^2 + \| H_2^{\text{obs}} - T_{\nu_k} \mu_{\alpha,2} \|^2 \right] \leq \sum_k \left[ \| H_1^{\text{obs}} - T_{\nu_k} \mu_{\alpha,2} \|^2 + \| H_2^{\text{obs}} - T_{\nu_k} \mu_{\alpha,1} \|^2 \right].
	\]
	which implies
	\[
	\alpha \| \mu_{\alpha,1} - \mu_{\alpha,2} \|^2 \leq \sum_{k=1}^K \langle H_1^{\text{obs}} - H_2^{\text{obs}}, T_{\nu_k} (\mu_{\alpha,1} - \mu_{\alpha,2}) \rangle.
	\]
	By using Cauchy-Schwarz inequality, it can be writen as
	\[
	\alpha \| \mu_{\alpha,1} - \mu_{\alpha,2} \|^2 \leq \| H_1^{\text{obs}} - H_2^{\text{obs}} \| \cdot \left( \sum_{k=1}^K \| T_{\nu_k} \| \right) \| \mu_{\alpha,1} - \mu_{\alpha,2} \|.
	\]
	then
	\[
	\| \mu_{\alpha,1} - \mu_{\alpha,2} \| \leq \frac{1}{\alpha} \left( \sum_{k=1}^K \| T_{\nu_k} \| \right) \| H_1^{\text{obs}} - H_2^{\text{obs}} \| = C(\alpha) \delta.
	\]
	
	\textit{Discrete case:} The solution is $\mu_\alpha = (A^T A + \alpha I)^{-1} A^T \mathbf{H}^{\text{obs}}$, so
	\[
	\| \mu_{\alpha,1} - \mu_{\alpha,2} \| \leq \| (A^T A + \alpha I)^{-1} A^T \| \cdot \| \mathbf{H}^{\text{obs}}_1 - \mathbf{H}^{\text{obs}}_2 \|_2 \leq \frac{\sigma_{\max}(A)}{\alpha} \delta.
	\]
	In fact, it can be proved by Singular Value Decomposition
	\[
	\|(A^T A + \alpha I)^{-1} A^T\| \leq \frac{1}{2\sqrt{\alpha}}.
	\]
	Therefore,
	\[
	\|\mu_{\alpha,1} - \mu_{\alpha,2}\|_2 \leq \frac{\delta}{2\sqrt{\alpha}}.
	\]
\end{proof}

The following convergence theorem states that when the noise level $\delta \to 0$, by reasonably choosing the regularization parameter $\alpha(\delta)$, the solution $\mu_\alpha$ of Tikhonov regularization will converge to the true solution $\mu^\dagger$ (Stability) and the convergence rate can reach $O(\sqrt{\delta})$ (Optimality).
\begin{theorem}[Convergence]
	\label{thm:convergence}
	Let $\mu^{\dagger}$ be the true solution. For noisy data $H^{\text{obs}} = T_{\nu_k} \mu^{\dagger} + \eta_k$ with $\sum_{k=1}^K \| \eta_k \|_{L^2}^2 \leq \delta^2$ (continuous) or $\| \mathbf{H}^{\text{obs}} - A \mu^{\dagger} \|_2 \leq \delta$ (discrete). If $\alpha(\delta) \to 0$ and $\delta^2 / \alpha(\delta) \to 0$ as $\delta \to 0$, then
	\[
	\mu_{\alpha(\delta)} \to \mu^{\dagger} \quad \text{as} \quad \delta \to 0.
	\]
	in $L^2([0,1])$ (continuous) or $\mathbb{R}^N$ (discrete). If the source condition $\mu^{\dagger} = ( \sum_{k=1}^K T_{\nu_k}^* T_{\nu_k} )^{1/2} w$ holds for some $w \in L^2([0,1])$, then
	\[
	\| \mu_\alpha - \mu^{\dagger} \|_{L^2} = O(\sqrt{\delta}).
	\]
\end{theorem}

\begin{proof}
	\textit{Continuous case:} The minimizer $\mu_\alpha$ satisfies the Euler-Lagrange equation
	\[
	\sum_{k=1}^K T_{\nu_k}^* (T_{\nu_k} \mu_\alpha - H^{\text{obs}}_k) + \alpha \mu_\alpha = 0.
	\]
	Substituting $H^{\text{obs}}_k = T_{\nu_k} \mu^{\dagger} + \eta_k$, we obtain
	\[
	\sum_{k=1}^K T_{\nu_k}^* T_{\nu_k} (\mu_\alpha - \mu^{\dagger}) + \alpha \mu_\alpha = \sum_{k=1}^K T_{\nu_k}^* \eta_k.
	\]
	Thus,
	\[
	\mu_\alpha - \mu^{\dagger} = \left( \sum_{k=1}^K T_{\nu_k}^* T_{\nu_k} + \alpha I \right)^{-1} \left( \sum_{k=1}^K T_{\nu_k}^* \eta_k - \alpha \mu^{\dagger} \right).
	\]
	Taking norms and using the source condition
	\begin{align*}
		\| \mu_\alpha - \mu^{\dagger} \| &\leq \left\| \left( \sum_{k} T_{\nu_k}^* T_{\nu_k} + \alpha I \right)^{-1} \sum_{k} T_{\nu_k}^* \eta_k \right\| + \alpha \left\| \left( \sum_{k} T_{\nu_k}^* T_{\nu_k} + \alpha I \right)^{-1} \mu^{\dagger} \right\| \\
		&\leq \frac{\delta}{\sqrt{\alpha}} + \alpha \| (\sum_{k} T_{\nu_k}^* T_{\nu_k})^{1/2} w \| \cdot \left\| (\sum_{k} T_{\nu_k}^* T_{\nu_k} + \alpha I)^{-1} \right\| \\
		&\leq \frac{\delta}{\sqrt{\alpha}} + \sqrt{\alpha} \| w \|.
	\end{align*}
	Choosing $\alpha(\delta) \sim \delta$ gives $\| \mu_\alpha - \mu^{\dagger} \| = O(\sqrt{\delta})$.
	
	\textit{Discrete case:} Analogous with singular value decomposition of $A$(cf.\cite{hansen2010discrete}). For the discrete regularization problem
	\[
	\min_{\mu \in \mathbb{R}^N} J(\mu) = \|A\mu - \mathbf{H}^{\text{obs}}\|_2^2 + \alpha\|\mu\|_2^2.
	\]
	where \(A = [A^{(1)}; \dots ; A^{(K)}] \in \mathbb{R}^{MK \times N}\), noisy data \( \mathbf{H}^{\text{obs}} = A\mu^\dagger + \eta \), the noise satisfies \( \|\eta\|_2 \leq \delta \).
	
	Expanding the objective function into a quadratic form
	\[
	J(\mu) = (A\mu - \mathbf{H}^{\text{obs}})^T (A\mu - \mathbf{H}^{\text{obs}}) + \alpha \mu^T \mu.
	\]
	by computing the derivative with respect to $ \mu $
	
	\[
	\nabla J(\mu) = 2A^T(A\mu - \mathbf{H}^{\text{obs}}) + 2\alpha\mu.
	\]
	then, we get
	\[
	A^TA\mu + \alpha\mu = A^T\mathbf{H}^{\text{obs}}.
	\]
	namely,
	\[ \mu_\alpha = (A^T A + \alpha I)^{-1} A^T \mathbf{H}^{\text{obs}}. \]
	
	Perform Singular Value Decomposition on \( A \)
	\[
	A = U\Sigma V^T, \quad \Sigma = \text{diag}(\sigma_1, \dots, \sigma_r), \quad \sigma_1 \geq \dots \geq \sigma_r > 0.
	\]
	then the explicit solution can be expressed as
	\[
	\mu_\alpha = V(\Sigma^T \Sigma + \alpha I)^{-1} \Sigma^T U^T \mathbf{H}^{\text{obs}}.
	\]
	and we can take the component form
	\[
	\mu_\alpha = \sum_{i=1}^r \frac{\sigma_i}{\sigma_i^2 + \alpha} \left( u_i^T \mathbf{H}^{\text{obs}} \right) v_i.
	\]
	
\end{proof}

\begin{remark}
	Multi-frequency data improves the condition number. The minimum singular value $\sigma_{\min}(A)$ of the global matrix satisfies
	\[
	\sigma_{\min}(A) \geq \min_k \sigma_{\min}(A^{(k)}).
	\]
	thus, the condition number \( \kappa(A) = \sigma_{\max}(A)/\sigma_{\min}(A) \) decreases, and the stability is enhanced.
	
	In fact, for any unit vector \( x \in \mathbb{R}^n \) (i.e., \( \|x\|_2 = 1 \)),
	\[
	\|Ax\|_2^2 = \sum_{k=1}^K \|A^{(k)}x\|_2^2 \geq \sum_{k=1}^K \left( \min_{\|y\|_2 = 1} \|A^{(k)}y\|_2 \right)^2 = \sum_{k=1}^K \sigma_{\min}(A^{(k)})^2.
	\]
	therefore,
	\[
	\sigma_{\min}(A) = \min_{\|x\|_2 = 1} \|Ax\|_2 \geq \sqrt{\sum_{k=1}^K \sigma_{\min}(A^{(k)})^2}.
	\]
\end{remark}

The above analysis established a rigorous theoretical framework for solving ill-posed 1D Helmholtz inverse problems using multi-frequency data fusion and Tikhonov regularization. Key contributions include
\begin{itemize}[left=5pt]
	\item Characterization of ill-posedness through compact operator theory (Theorem \ref{thm:illposedness})
	\item Uniqueness guarantees via dense frequency sampling (Theorem \ref{Uniqueness})
	\item Well-posedness of regularized formulation (Theorems \ref{thm:existence}, \ref{thm:stability}, \ref{thm:convergence})
\end{itemize}

The framework provides mathematical foundations for stable reconstruction of $\mu$ from multi-frequency measurements. Future work includes convergence rate analysis, optimal parameter selection, and extensions to higher dimensions.

\subsubsection{Numerical experiment results}
After conducting numerical experiments by incorporating the strength functions from the literature \cite{li2020inverse} and defining more complex strength functions, the following presents the numerical case results and a brief analysis.\\
i) True strength function $\mu(y)=0.5(1-\cos(2\pi y))$, ii) True strength function $\mu(y)=0.6-0.3\cos(2\pi y)-0.3\cos(4\pi y)$.
\begin{figure}[ht]
	\centering
	\subfigure[$\nu = 1:3$]{\includegraphics[width=0.45\textwidth]{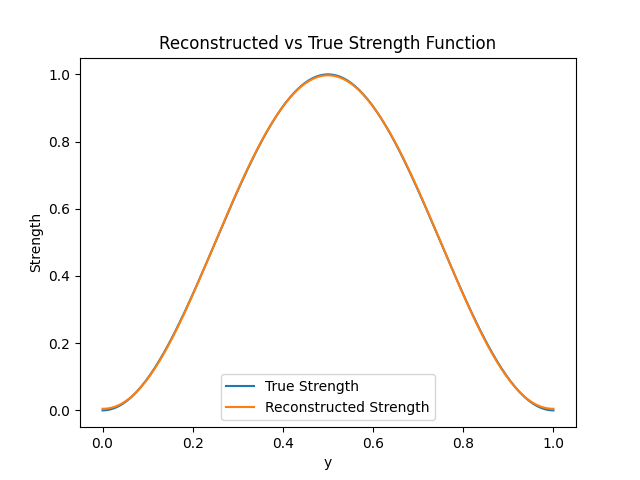}}
	\hspace{1cm}
	\subfigure[$\nu = 1:4$]{\includegraphics[width=0.45\textwidth]{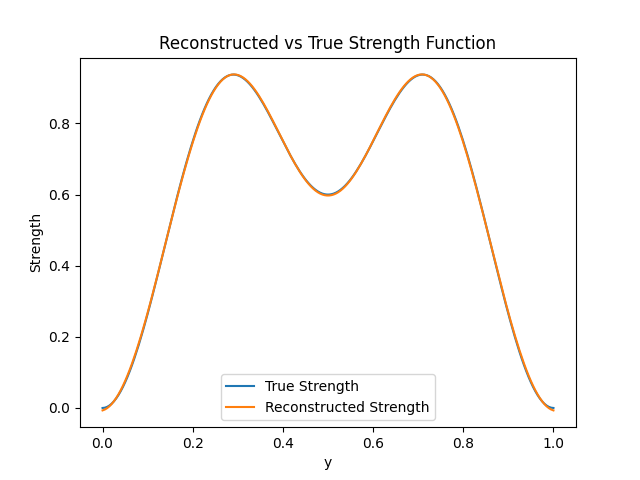}}
	\caption{Single Fourier mode}
	\label{f1}  
\end{figure}

i) and ii) correspond to (a) and (b) in Figure \ref{f1}. For simple real intensity functions with a single Fourier mode, it can be seen that three to four frequency data points produce a good reconstruction result.

\noindent
iii) True strength function $\mu(y) = 0.5e - 0.3e^{\cos(4\pi y)} - 0.2e^{\cos(6\pi y)}$, iv) True strength function $\mu(y) = 0.5e^{\cos(6\pi y)} - 0.3e^{\sin(8\pi y)}$, v) True strength function $\mu(y) = 0.6e - 0.5e^{\cos(6\pi y)} - 0.3e^{\sin(8\pi y)}$.\\
\begin{figure}[h]
	\centering
	\subfigure[$\nu = 2:11$]{\includegraphics[width=0.45\textwidth]{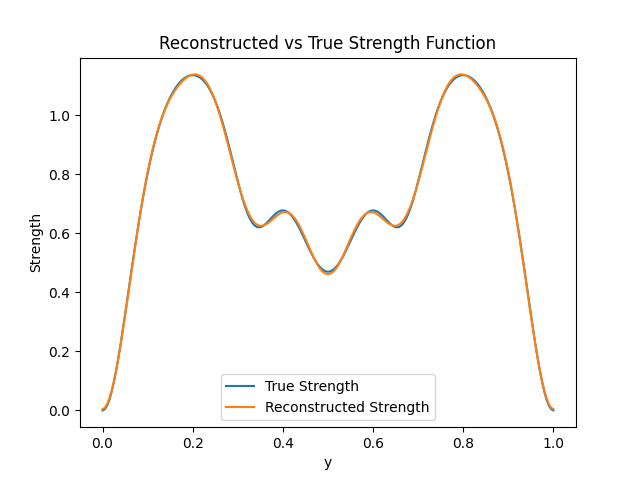}}
	\hspace{2cm}
	\subfigure[$\nu = 2:11$]{\includegraphics[width=0.45\textwidth]{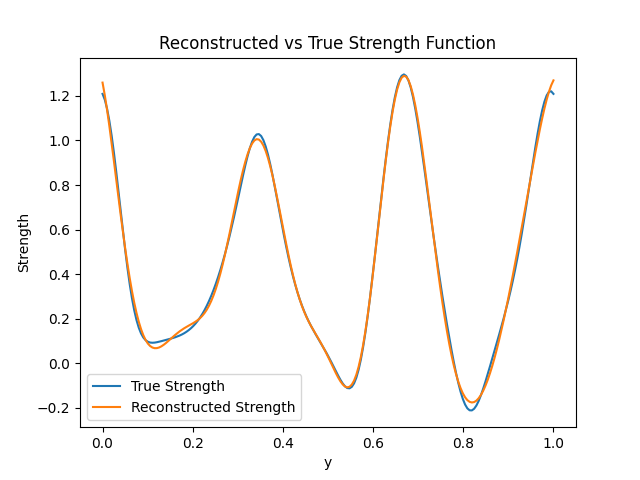}}
	\subfigure[$\nu = 2:11$]{\includegraphics[width=0.45\textwidth]{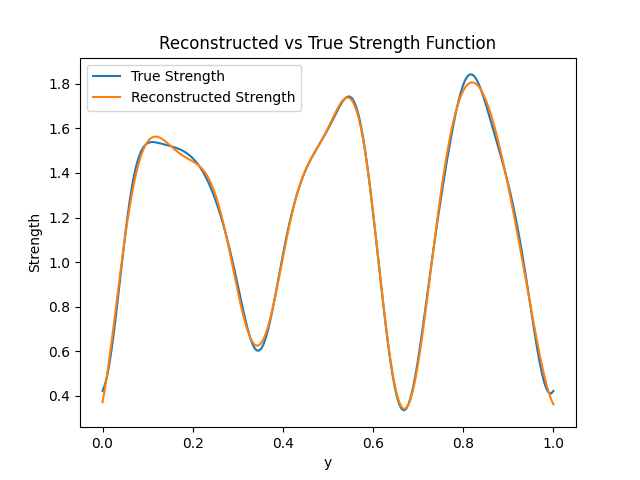}}
	\caption{Multiple Fourier mode}
	\label{f2}  
\end{figure}

iii), iv), and v) correspond to (a), (b), and (c) in Figure \ref{f2}. For more complex real strength functions with multiple Fourier modes, approximately 10 frequency data points are sufficient to achieve good reconstruction results.

Final numerical results analysis: for strength functions with a single Fourier mode, only 2 to 4 frequency data points are needed to achieve good reconstruction results without attenuation; for strength functions with multiple Fourier modes, only 10 frequency data points are needed to achieve good reconstruction results without attenuation. Compared to the 16 frequency data points used in \cite{li2020inverse} under the $\sigma>0$ scenario, there has been a significant improvement. By incorporating CPU parallel computing, the speed of data generation and optimization processes for different frequency data points has been enhanced, resulting in a substantial increase in program runtime efficiency. In summary, the combined use of multi-frequency data fusion and regularization methods has to some extent mitigated the ill-posedness issue of the one-dimensional stochastic Helmholtz equation without attenuation.

\section{Stochastic wave equations driven by finite-jump Lévy processes}
In this section, we analyze the direct problem of a stochastic wave equation driven by a finite-jump Lévy process and obtain stability estimates for its mild solutions. Subsequently, we discuss the inverse problem and reconstruct the source terms $f(x)$ and $g(x)$ based on a spectral decomposition method using global modes.

\subsection{Direct problem}
\begin{theorem}
	If $f,g\in L^2(D)$ and $\|g\|_{L^2(D)}\neq 0$, under the condition that the support of the non-negative function $h\in L^{\infty}(0,T)$ is a positive measure, the mild solution $(\ref{mild})$ will satisfy certain stability estimate
	\begin{equation}
		\begin{aligned}
			E \left[ \left\| u \right\|^2_{L^2(D \times [0,T])} \right] &\leq \frac{2T^4}{3} \| h \|^2_{L^\infty} \| f \|^2_{L^2(D)} + \left( \frac{2T^4 b^2 + T^3 \cdot \sigma^2 + T^3 \lambda_p \sigma_j^2}{3} \right) \| g \|^2_{L^2(D)}\\
			&\leq C \left( \| f \|^2_{L^2(D)} + \| g \|^2_{L^2(D)} \right).
		\end{aligned}
	\end{equation}
	where $C>0$ depends on time $T$, drift coefficient $b$, diffusion coefficient $\sigma$, and jump parameters $\lambda_{p}$ and $\sigma_{j}^2$.
\end{theorem}

\begin{proof}
	$a)$ In the stability estimate, $u_{\text{det}}(x,t)$ is expanded in terms of the characteristic function system of the Laplace operator
	\[ u_{\text{det}}(x,t) = \sum_{k=1}^{\infty} \left( \int_0^t \frac{\sin\left( (t-\tau)\sqrt{\lambda_k} \right)}{\sqrt{\lambda_k}} \left( f_k h(\tau) + g_k b \right) d\tau \right) \varphi_k(x). \]
	where $f_k=\langle f,\varphi_{k}\rangle$ and $g_k=\langle g,\varphi_{k}\rangle$.
	The $L^2(D)$-norm of the solution is the sum of the squares of the contributions from each mode
	\[ \left\| u_{\text{det}} \right\|^2_{L^2(D)} = \sum_{k=1}^{\infty} \left( \int_0^t \frac{\sin\left( (t-\tau)\sqrt{\lambda_k} \right)} {\sqrt{\lambda_k}} \left( f_k h(\tau) + g_k b \right) d\tau \right)^2. \]
	
	Apply the Cauchy-Schwarz inequality to the integral for each mode
	\[ \left( \int_0^t \frac{\sin\left( (t-\tau)\sqrt{\lambda_k} \right)}{\sqrt{\lambda_k}} \left( f_k h(\tau) + g_k b \right) d\tau \right)^2 \leq \left( \int_0^t \frac{\sin^2\left( (t-\tau)\sqrt{\lambda_k} \right)}{\lambda_ k} d\tau \right) \left( \int_0^t \left( f_k h(\tau) + g_k b \right)^2 d\tau \right). \]
	Estimating the first term on the right-hand side of the inequality, since for any real number $x$, $\sin^2(x) \leq x^2$, we get
	\[ \int_0^t \frac{\sin^2\left( (t-\tau)\sqrt{\lambda_k} \right)}{\lambda_k} d\tau \leq \int_0^t (t-\tau)^2 d\tau = \frac{t^3}{3}. \]
	Estimating the second term, since $(a+b)^2\leq 2(a^2+b^2)$, we have
	\[ \int_0^t \left( f_k h(\tau) + g_k b \right)^2 d\tau \leq 2 \left( f_k^2 \| h \|_{L^\infty}^2 \int_0^t d\tau + g_k^2 b^2 \int_0^t d\tau \right) = 2t \left( f_k^2 \| h \|_{ L^\infty}^2 + g_k^2 b^2 \right). \]
	Substituting this gives the estimate for each mode
	\[ \left( \int_0^t \frac{\sin\left( \cdot \right)}{\lambda_k} \left( f_k h + g_k b \right) d\tau \right)^2 \leq \frac{t^3}{3} \cdot 2t \left( f_k^2 \| h \|_{L^\infty}^2 + g_k^2 b^2 \right) = \frac{ 2t^4}{3} \left( f_k^2 \| h \|_{L^\infty}^2 + g_k^2 b^2 \right). \]
	Take the summation over all modes
	\[ \left\| u_{\text{det}} \right\|^2_{L^2(D)} \leq \frac{2t^4}{3} \sum_{k=1}^{\infty} \left( \| h \|_{L^\infty}^2 f_k^2 + b^2 g_k^2 \right). \]
	Furthermore,
	\[ \left\| u_{\text{det}} \right\|^2_{L^2(D)} \leq \frac{2t^4}{3} \left( \| h \|_{L^\infty}^2 \sum_{k=1}^{\infty} f_k^2 + b^2 \sum_{k=1}^{\infty} g_k^2 \right). \]
	and
	\[ \sum_{k=1}^{\infty} f_k^2 = \| f \|^2_{L^2(D)}, \quad \sum_{k=1}^{\infty} g_k^2 = \| g \|^2_{L^2(D)}. \]
	Finally, we obtain
	\[ \left\| u_{\text{det}} \right\|^2_{L^2(D)} \leq \frac{2t^4}{3} \left( \| h \|_{L^\infty}^2 \| f \|^2_{L^2(D)} + b^2 \| g \|^2_{L^2(D)} \right). \]
	Let \(C_1(T) = \frac{2t^4}{3} \max \left( \| h \|_{L^\infty}^2, b^2 \right)\), then we obtain
	\[ \left\| u_{\text{det}} \right\|^2_{L^2(D)} \leq C_1(T) \left( \| f \|^2_{L^2(D)} + \| g \|^2_{L^2(D)} \right). \]
	$b)$ The solution to the diffusion part is $u_{\text{diff}}(x, t) = \sigma \int_0^t K(x,t-\tau) g(x) dW_\tau$, where the kernel function is \[ K(x,t-\tau) = \sum_{\substack{k=1}}^{\infty} \frac{\sin \left( (t-\tau) \sqrt{\lambda_k} \right)}{\sqrt{\lambda_k}} \varphi_k(x). \]
	\[ \| u_{\text{diff}} \|^2_{L^2(D)} = \int_D \left( \sigma \left( \int_0^t K(x,t-\tau) g(x) dW_{\tau} \right)^2 \right) dx. \]
	Using Fubini's theorem, we have
	\[ E\left[ \| u_{\text{diff}} \|^2_{L^2(D)} \right] = \sigma^2 \int_D E\left[ \left( \int_0^t K(x,t-\tau) g(x) dW_{\tau} \right)^2 \right] dx. \]
	Since $E\left[ \left( \int_0^t H_{\tau} dW_{\tau} \right)^2 \right] = E\left[ \int_0^t H_{\tau}^2 d\tau \right]$, where $H_{\tau} = K(x,t-\tau) g(x)$, we have
	\[ E\left[ \left( \int_0^t K(x,t-\tau) g(x) dW_{\tau} \right)^2 \right] = E\left[ \int_0^t K(x,t-\tau)^2 g(x)^2 d\tau \right]. \]
	Since $K(x,t-\tau)^2 g(x)^2$ is deterministic, we get
	\[ E\left[ \int_0^t K(x,t-\tau)^2 g(x)^2 d\tau \right] = \int_0^t K(x,t-\tau)^2 g(x)^2 d\tau. \]
	Thus, \[ E\left[ \| u_{\text{diff}} \|^2_{L^2(D)} \right] = \sigma^2 \int_D \int_0^t K(x,t-\tau)^2 g(x)^2 d\tau dx. \]
	By interchanging the order of integration, we have
	\[ E\left[ \| u_{\text{diff}} \|^2_{L^2(D)} \right] = \sigma^2 \int_0^t \int_D K(x,t-\tau)^2 g(x)^2 dx d\tau. \]
	The inner integral is the square of the $L^2(D)$-norm of the kernel function and $g(x)$
	\[ E\left[ \| u_{\text{diff}} \|^2_{L^2(D)} \right] = \sigma^2 \int_0^t \| K(\cdot,t-\tau) g(\cdot) \|^2_{L^2(D)} d\tau. \]
	Expanding $g(x)$ in terms of the characteristic function system, we have
	\[ g(x) = \sum_{k=1}^{\infty} g_k \varphi_k(x), \quad g_k = \langle g, \varphi_k \rangle_{L^2(D)}. \]
	Then the product of the kernel function and $g(x)$ is
	\[ K(x,t-\tau)g(x) = \sum_{k=1}^{\infty} \frac{\sin\left((t-\tau)\sqrt{\lambda_k}\right)}{\sqrt{\lambda_k}} g_k \varphi_k(x). \]
	Using the orthogonality $\langle \varphi_k, \varphi_l \rangle = \delta_{kl}$, we obtain
	\[ \left\| K(\cdot, t - \tau)g(\cdot) \right\|^2_{L^2(D)} = \sum_{k=1}^{\infty} \left( \frac{\sin\left( (t - \tau) \sqrt{\lambda_k} \right)}{\sqrt{\lambda_k}} g_k \right)^2. \]
	Substituting the expected expression, we obtain
	\[ E\left[ \| u_{\text{diff}} \|^2_{L^2(D)} \right] = \sigma^2 \int_0^t \sum_{k=1}^{\infty} \frac{\sin^2\left( (t - \tau) \sqrt{\lambda_k} \right)}{\lambda_k} g_k^2 \, d\tau. \]
	An upper bound estimate is obtained for each k term
	\[ \int_0^t \frac{\sin^2\left( (t - \tau) \sqrt{\lambda_k} \right)}{\lambda_k} \, d\tau \leq \int_0^t (t - \tau)^2 \, d\tau = \frac{t^3}{3}. \]
	Therefore, \[ E\left[ \| u_{\text{diff}} \|^2_{L^2(D)} \right] \leq \frac{t^3 \cdot \sigma^2}{3} \sum_{k=1}^{\infty} g_k^2 = \frac{t^3 \cdot \sigma^2}{3} \| g \|^2_{L^2(D)}. \]
	In summary, there exists a constant $C_2$ related to $\sigma$ and $t$ such that
	\[ E\left[ \| u_{\text{diff}} \|^2_{L^2(D)} \right] \leq C_2 \| g \|^2_{L^2(D)}. \]
	
	For convenience, the following assumptions are needed.
	\begin{Assumption}
		Assume that the jump amplitude $J_s$ of the finite-jump Lévy process $L_t$ follows a normal distribution, and the jump time interval $s_j$ follows a Poisson distribution. i.e.,
		\begin{itemize}[left=5pt]
			\item The jump magnitudes $\left\{ J_{S_j} \right\}_{j=1}^{N_T}$ are independent and identically distributed, following $N(0,\sigma_{j}^2)$.
			\item The jump times $\left\{ S_j \right\}_{j=1}^{N_T}$ are spaced according to a Poisson distribution with parameter $\lambda_p$. (i.e., the expected number of jumps per unit time is $\lambda_p$)
		\end{itemize}
		At this point, the Lévy process decomposes into
		\[ L_t = bt + \sigma W_t + \sum_{j=1}^{N_t} J_{S_j}. \]
		where $N_t \sim \text{Poisson}(\lambda_p t)$ and $J_{S_j} \overset{\text{i.i.d.}}{\sim} N(0, \sigma_j^2)$.\\
		\label{assum:2}
	\end{Assumption}\noindent
	$c)$ Since the jump times are a random point process, the solution requires integration over the Poisson process, and equation (\ref{jump}) becomes
	\[ u_{\text{jump}}(x,t) = \sum_{j=1}^{N_t} K(x,t - s_j) g(x) J_{s_j} = \int_0^t K(x,t - \tau) g(x) \sum_{j=1}^{N_\tau} J_{s_j} \delta_{s_j} (d\tau). \]
	Thus, \[ \| u_{\text{jump}} \|^2_{L^2(D)} = \int_D \left( \sum_{j=1}^{N_t} \sum_{k=1}^{\infty} \frac{\sin\left( (T - s_j) \sqrt{\lambda_k} \right)}{\sqrt{\lambda_k}} g_k J_{s_j} \varphi_k(x) \right)^2 dx. \]
	Since the characteristic functions $\{\varphi_k\}$ are orthogonal, the cross terms integrate to zero, simplifying to
	\[ \left\| u_{\text{jump}} \right\|^2_{L^2(D)} = \sum_{k=1}^{\infty} \sum_{j=1}^{N_l} \left( \frac{\sin \left( \left( T - s_j \right) \sqrt{\lambda_k} \right)}{\sqrt{\lambda_k}} g_k J_{s_j} \right)^2. \]
	Taking the expectation and expanding the square
	\[ E\left[ \left\| u_{\text{jump}} \right\|^2_{L^2(D)} \right] = \sum_{k=1}^{\infty} E \left[ \sum_{j=1}^{N_l} \frac{\sin^2 \left( \left( T - s_j \right) \sqrt{\lambda_k} \right)}{\lambda_k} g_k^2 J_{s_j }^2 \right]. \]
	For the Poisson process $N_t$, by Campbell's theorem, the stochastic sum is converted to an integral
	\[ E \left[ \sum_{j=1}^{N_t} F(s_j) \right] = \lambda_p \int_0^t F(s) \, ds. \]
	where $F(s) = \frac{\sin^2 \left( (T - s_j) \sqrt{\lambda_k} \right)}{\lambda_k} g_k^2 E[J_{s_j}^2]$, substituting gives
	\[ E \left[ \left\| u_{\text{jump}} \right\|^2_{L^2(D)} \right] = \sum_{k=1}^{\infty} \lambda_p \sigma_j^2 \int_0^t \frac{\sin^2 \left( (T - s_j) \sqrt{\lambda_k} \right)}{\lambda_k} g_k^2 \, ds. \]
	Since $\int_0^t \frac{\sin^2 \left( (t-s) \sqrt{\lambda_k} \right)}{\lambda_k} \, ds \leq \int_0^t (t-s)^2 \, ds = \frac{t^3}{3}$, we obtain
	\[ E\left[ \left\| u_{\text{jump}} \right\|^2_{L^2(D)} \right] \leq \frac{t^3}{3} \lambda_p \sigma_j^2 \| g \|^2_{L^2(D)}. \]
	Namely, there exist constants $C_3$ related to $\sigma_{j}^2$, $\lambda_{p}$, and $t$ such that
	\[ E\left[ \left\| u_{\text{jump}} \right\|^2_{L^2(D)} \right] \leq C_3 \| g \|^2_{L^2(D)}. \]
	By (a), (b), and (c), we have
	Under the assumptions that $f,g\in L^2(D) $ and $b,\sigma\in \mathbb{R}$, the mild solution satisfies the following stability estimate
	\begin{align*}
		E \left[ \left\| u \right\|^2_{L^2(D \times [0,T])} \right] &\leq \frac{2T^4}{3} \| h \|^2_{L^\infty} \| f \|^2_{L^2(D)} + \left( \frac{2T^4 b^2 + T^3 \cdot \sigma^2 + T^3 \lambda_p \sigma_j^2}{3} \right) \| g \|^2_{L^2(D)}\\
		&\leq C \left( \| f \|^2_{L^2(D)} + \| g \|^2_{L^2(D)} \right).
	\end{align*}	
\end{proof}

In the following subsection, we will discuss the inverse problem.
\subsection{Inverse Problem}
For the ill-posedness of the inverse problem of a stochastic wave equation driven by a finite-jump Lévy process, it is necessary to design a regularization method. Currently, we consider handling the deterministic part by observing the mean (expectation) of the final-time data and handling the random source term by observing the covariance of the final-time data. This section analyzes the reconstruction of the deterministic term $f(x)$ and the random source term $g(x)$.

\subsubsection{Reconstruction of the deterministic term $f(x)$}
Expand the deterministic part into a set of eigenfunctions of the Laplace operator
\[ u_{\text{det}}(x,T) = \sum_{k=1}^{\infty} \left( f_k \int_0^T h(\tau) \frac{\sin((T-\tau)\sqrt{\lambda_k})}{\sqrt{\lambda_k}} d\tau + g_k b \int_0^T \frac{\sin((T-\tau)\sqrt{\lambda_k})}{\sqrt{\lambda_k}} d\tau \right) \varphi_k(x). \]
Using the uniqueness of the linear part: if $h(t)$ satisfies the non-degenerate condition (e.g., $h(t)$ is not zero everywhere on $[0,T]$), the reconstruction equation for the spectral coefficients
\[ E[u_k(T)] = f_k \int_0^T h(\tau) \frac{\sin((T-\tau)\sqrt{\lambda_k})}{\sqrt{\lambda_k}} d\tau + g_k b \int_0^T \frac{\sin((T-\tau)\sqrt{\lambda_k})}{\sqrt{\lambda_k}} d\tau. \]
If \(b = 0\) (i.e., there is no drift term in the Lévy process), then
\[ E[u_k(T)] = f_k \cdot \int_0^T h(\tau) \frac{\sin((T-\tau)\sqrt{\lambda_k})}{\sqrt{\lambda_k}} d\tau. \]
So we directly recover
\begin{equation}
	f_k = \frac{E[u_k(T)]}{\int_0^T h(\tau) \frac{\sin((T-\tau)\sqrt{\lambda_k})}{\sqrt{\lambda_k}} d\tau}.
	\label{f}
\end{equation}

\subsubsection{Reconstruction of random source term $g(x)$}
The linear property of the covariance function indicates that
\begin{equation*}
	\text{Cov}(u(x,T),u(y,T))=E\left[ (u-E[u])(u-E[u])^T \right].
\end{equation*}
Substituting the decomposed solution
\begin{align*}
	\text{Cov}(u(x,T), u(y,T)) &= E\left[(u_{\text{diff}} + u_{\text{jump}})(u_{\text{diff}} + u_{\text{jump}}) \right] \\
	&= E[u_{\text{diff}}(x,T) u_{\text{diff}}(y,T)] + E[u_{\text{jump}}(x,T) u_{\text{jump}}(y,T)].
\end{align*}

In fact, the cross term $E[u_{\text{diff}} u_{\text{jump}}]$
is zero, since the diffusion term $u_{\text{diff}}$ is driven by Brownian motion and the jump term $u_{\text{jump}}$ is driven by a jump process, and the two random sources are independent. Furthermore, the increments of Brownian motion and the jump process both have zero mean (even if the jump amplitude $J_{s}$ is non-zero, it is canceled out by centralization when calculating the covariance).

Assuming that the diffusion coefficient $\sigma $ and jump parameter $J_{s}$ are known, and by the properties of the Itô integral, the increments of Brownian motion at different time points are independent and orthogonal, then the covariance structure of the diffusion component
\begin{equation*}
	E\left[u_{\text{diff},k}(T) u_{\text{diff},l}(T)\right] = \sigma^2 E\left[\int_0^T \dfrac{\sin\left((T-\tau) \sqrt{\lambda_k}\right)}{\sqrt{\lambda_k}}   g_k dW_{\tau} \cdot \int_0^T \dfrac{\sin\left((T-\tau) \sqrt{\lambda_l}\right)}{\sqrt{\lambda_l}} g_l dW_{\tau}\right].
\end{equation*}
By the Itô equidistality and the independent incrementality of Brownian motion, the right-hand side of the above equation will be equal to
\begin{equation*}
	\sigma^2 g_k g_l \dfrac{1}{\sqrt{\lambda_k \lambda_l}}  \int_0^T \sin\left((T-\tau) \sqrt{\lambda_k}\right) \sin\left((T-\tau) \sqrt{\lambda_l}\right) d\tau.
\end{equation*}
Covariance structure of the jump part
\begin{align*}
	E[u_{\text{jump},k}(T) u_{\text{jump},l}(T)] &= \sum_{j=1}^{N_T} E\left[J_{s_j}^2 \frac{\sin((T - s_j) \sqrt{\lambda_k})}{\sqrt{\lambda_k}} \cdot \frac{\sin((T - s_j) \sqrt{\lambda_l})}{\sqrt{\lambda_l}} \right] g_k \cdot g_l \\ 
	&= \lambda_p E[J^2] \int_0^T \frac{\sin((T - \tau) \sqrt{\lambda_k})}{\sqrt{\lambda_k}} \cdot \frac{\sin((T - \tau) \sqrt{\lambda_l})}{\sqrt{\lambda_l}} g_k g_l d\tau.  
\end{align*}
Consequently,
\begin{align*}
	&\text{Cov}(u_k, u_l)\\ &= \sigma^2 g_k g_l \int_0^T \frac{\sin\left((T-\tau) \sqrt{\lambda_k}\right) \cdot \sin\left((T-\tau) \sqrt{\lambda_l}\right)}{\sqrt{\lambda_k \lambda_l}} d\tau + \lambda_p \sigma_j^2 \int_0^T \frac{\sin\left((T-\tau) \sqrt{\lambda_k}\right) \cdot \sin\left((T-\tau) \sqrt{\lambda_l}\right)}{\sqrt{\lambda_k \lambda_l}} g_k g_l d\tau \\
	&= g_k g_l \cdot (\sigma^2 + \lambda_p \sigma_j^2) \cdot I_{kl}.
\end{align*}
where $I_{kl}=\int_0^T \dfrac{\sin\left((T-\tau) \sqrt{\lambda_k}\right)\cdot \sin\left((T-\tau)\sqrt{\lambda_l}\right)}{\sqrt{\lambda_k \lambda_l}} d\tau$, thus
\begin{equation}
	g_k g_l = \frac{\text{Cov}(u_k, u_l)}{(\sigma^2 + \lambda_p \sigma_j^2) \cdot I_{kl}}
	\label{g}.
\end{equation}
The following theorem discusses the uniqueness of the inverse source problem.

\begin{theorem}
	Assume that Assumption $\ref{assum:2}$ holds. \\
	$(1)$ If $h(t)$ is a monotonically increasing function and $T$ is a rational number, or if $h(t)$ is a strictly monotonically increasing function and $T$ is a real number. then $f$ is uniquely determined by $\{E[u_k(T)]; k\in \mathbb{N}\}$. \\
	$(2)$ If $T$ is an arbitrary algebraic number, then the source term $g$ is uniquely determined by $\text{Cov}(u_k(T), u_l(T)); k,l\in \mathbb{N}$.
\end{theorem}
\begin{proof}
	The proof process for \textbf{(1)} is similar to the discussion of equation $(4.1)$ in Section 4.1 of \cite{feng2021inverse}, so it will not be repeated here. For the discussion of case \textbf{(2)}, we know that the eigenvalues and eigenvectors of the Laplace operator $-\Delta$
	\[ \lambda_k = k^2, \quad \varphi_k = \sqrt{\frac{2}{\pi}} \sin(kx), \quad k = 1, 2, \dots \]
	
	Considering the integral definition
	\begin{equation*}
		I_{kl} = \int_0^T \dfrac{\sin\left(k(T - \tau)\right)}{k}  \cdot \dfrac{\sin\left(l(T - \tau)\right)}{l} d\tau.
	\end{equation*}
	The simplified equivalent form (using the substitution $s=T - \tau $, $ds= -d\tau$)
	\begin{equation*}
		I_{kl} = \frac{1}{kl} \int_0^T \sin(ks) \cdot \sin(ls) \, ds.
	\end{equation*}
	when $k=l$, the integral simplifies to
	\begin{equation*}
		I_{kk} =\dfrac{1}{k^2}  \int_0^T \sin^2(ks) \, ds = \dfrac{1}{k^2} \left[ \frac{T}{2} - \frac{\sin(2kT)}{4k}  \right].
	\end{equation*}
	when $k \neq l$, using the trigonometric identity $\sin A \sin B = \frac{1}{2} \left[ \cos(A - B) - \cos(A + B) \right]$, the integral becomes
	\begin{equation*}
		I_{kl} =\dfrac{1}{2kl}  \int_0^T \left[ \cos((k-l)s) - \cos((k+l)s) \right] \, ds.
	\end{equation*}
	After calculating, we obtain
	\begin{equation}
		I_{kl} =\dfrac{1}{2kl} \left[ \frac{\sin\left((k-l)T\right)}{k-l} - \frac{\sin\left((k+l)T\right)}{k+l} \right].
		\label{ikl}
	\end{equation}
	If $k>l$, it is clear that for any algebraic $T$, $I_{kl}\neq 0$. If $I_{kl}=0$, then
	\[ \frac{\sin\left((k-l)T\right)}{k-l} - \frac{\sin\left((k+l)T\right)}{k+l}=0. \]
	From $\sin \theta = \frac{e^{i\theta}-e^{-i\theta}}{2i}$, we have
	\[ (k + l) e^{i(k - l)T} - (k + l) e^{-i(k - l)T} - (k - l) e^{i(k + l)T} + (k - l) e^{-i(k + l)T} = 0 .\]
	Let there be a set of distinct integers $\{k_1,k_2,\dots,k_n\}$ and a set of distinct real parameters $\{T_1,T_2,\dots,T_n\}$ such that $T_j$ is an algebraically independent number (satisfying the conditions of the Lindemann–Weierstrass theorem, see [\cite{baker2022transcendental}]). Define the matrix
	\[ M = \left( e^{i k_j T_l} \right)_{1 \leq j,l \leq n} .\]
	by expanding
	\[M = 
	\begin{pmatrix}
		e^{i k_1 T_1} & e^{i k_1 T_2} & \cdots & e^{i k_1 T_n} \\
		e^{i k_2 T_1} & e^{i k_2 T_2} & \cdots & e^{i k_2 T_n} \\
		\vdots       & \vdots       & \ddots & \vdots \\
		e^{i k_n T_1} & e^{i k_n T_2} & \cdots & e^{i k_n T_n}
	\end{pmatrix}
	\]
	Let
	\[
	z_l := e^{i T_l} \quad (1 \leq l \leq n).
	\]
	then
	\[
	M_{jl} = z_l^{k_j}.
	\]
	Assume that $k_j=k_1+(j-1)$, i.e., a set of consecutive increasing integers. Under this assumption, $M$ can be written as
	\[
	M = D \cdot V, \quad \text{where} \quad D = \mathrm{diag}(z_1^{k_1}, z_2^{k_1}, \ldots, z_n^{k_1}), \quad V_{jl} = z_l^{j-1}.
	\]
	$V$ is the classical Vandermonde matrix
	\[
	V = 
	\begin{pmatrix}
		1 & 1 & \cdots & 1 \\
		z_1 & z_2 & \cdots & z_n \\
		z_1^2 & z_2^2 & \cdots & z_n^2 \\
		\vdots & \vdots & \ddots & \vdots \\
		z_1^{n-1} & z_2^{n-1} & \cdots & z_n^{n-1}
	\end{pmatrix}
	\]
	and vandermonde determinant formula
	\[
	\det(V) = \prod_{1 \leq i < j \leq n} (z_j - z_i).
	\]
	Therefore,
	\[
	\det(M) = \det(D) \cdot \det(V) = \left( \prod_{l=1}^{n} z_l^{k_1} \right) \cdot \prod_{1 \leq i < j \leq n} (z_j - z_i).
	\]
	Since $z_l = e^{i T_l}$ and $T_l$ are distinct, $z_l$ are also distinct. The Lindemann–Weierstrass theorem guarantees the independence of $T_l$, so there are no repeated terms in $z_j$. Therefore, the Vandermonde determinant $\prod_{i<j} (z_j - z_i) \ne 0$. And since each $z_l^{k_1} \ne 0$, the overall determinant $\det(M) \ne 0$, which contradicts $I_{kl}=0$. Finally, for any $k,l\in \mathbb{N}$ and any algebraic $T$, we have $I_{kl}\neq 0$.
	In fact, when $k\approx l$, the denominator $k - l$ causes numerical instability. To address this, consider using the complex exponential method:
	Using Euler's formula $\cos(\theta) = \Re\left[e^{i\theta}\right]$, the integral can be transformed into
	\begin{equation*}
		\int_0^T \cos(ms) \, ds = \Re\left[\int_0^T e^{ims} \, ds \right].
	\end{equation*}
	Calculating the complex integral, we obtain
	\begin{equation*}
		\int_0^T e^{ims} \, ds = \frac{e^{imT} - 1}{im}.
	\end{equation*}
	Expanding into real and imaginary parts
	\begin{equation*}
		\dfrac{e^{imT} - 1}{im} = \dfrac{\cos(mT) + i \sin(mT) - 1}{im} = \dfrac{\sin(mT)}{m} + \dfrac{1 - \cos(mT)}{m}i.
	\end{equation*}
	Thus,
	\begin{equation*}
		I_{kl} = \dfrac{1}{2kl} \left[ \dfrac{\text{Im}(z_1)}{k-l} - \dfrac{\text{Im}(z_2)}{k+l}  \right].
	\end{equation*}
	where $z_1 = e^{i(k-l)T} - 1, \quad z_2 = e^{i(k+l)T} - 1$, hence equation (\ref{ikl}) is proven.
\end{proof}

\begin{remark}
	This optimization essentially suppresses high-frequency errors by analyzing the integral expression and avoiding direct calculation of the difference of oscillatory terms, thereby improving the accuracy of covariance kernel calculations in the inverse problem and ultimately improving the recovery effect of random source terms.
\end{remark}
The following theorem illustrates the stability of the reconstruction process for $f$ and $g$.

\begin{theorem}
	The recovery of source functions $f$ and $g$ in the inverse source problem is unstable, and the following estimate holds
	\[ \left| \int_0^T \frac{h(\tau) \sin \left( \sqrt{\lambda_k} (T - \tau) \right)}{\sqrt{\lambda_k}} \, d\tau \right| \leq \frac{1}{\sqrt{\lambda_k}} \int_0^T h(\tau) \, d\tau \to 0 \quad \text{as} \quad  k \to \infty,
	\]
	\[ E\left[\int_0^T \frac{\sin^2\left( (T - \tau) \sqrt{\lambda_k} \right)}{\lambda_k} \, d\tau\right] \leq \frac{T}{\lambda_{k}}. \]
\end{theorem}

\subsubsection{Numerical experimental design}
In this subsubsection, we will implement a numerical experiment design for simulating the forward problem of a stochastic wave equation driven by a finite-jump Lévy process and reconstructing the source function in one dimension. In the space-time domain, we set $x \in [0, L], \, t \in [0, T], \, L = \pi, \, T = 1$. Specifically, we first define the spatial grid and time grid
\[x_i = i \Delta x, \quad i = 0, 1, 2, \dots, N_x, \quad \Delta x = \frac{L}{N_x}, \quad N_x = 100. \]
\[t_j = j \Delta t, \quad j = 0, 1, 2, \dots, N_t, \quad \Delta t = \frac{T}{N_t}, \quad N_t = 1000. \]

Before simulating the direct problem, we need to discretize the kernel function and define the calculation method for its integral weights. Firstly, we discretize the eigenfunctions. Since we are studying a one-dimensional wave equation (with a Lévy driving term) with Dirichlet boundary conditions in the spatial segment, a set of orthogonal eigenfunctions of the $Laplace$ operator can be taken as
$\varphi_k(x_i) = \sqrt{\frac{2}{L}} \sin(k x_i), \, k = 1, \dots, K, \, i = 0, \dots, N_x.$ satisfying\[ -\frac{d^2}{dx^2} \varphi_k(x) = k^2 \varphi_k(x), \quad \varphi_k(0) = \varphi_k(L) = 0. \]
Define the kernel function $A_k(t_j) = \frac{\sin(k (T - t_j))}{k}, \, j = 0, \dots, N_t$. The product-of-integrals weights (matrix $I_{kl}$) can be written as
\[I_{kl} = \frac{1}{T} \int_0^T A_k(\tau) A_l(\tau) d\tau \approx \frac{1}{T} \sum_{j=0}^{N_t-1} A_k(t_j) A_l(t_j) \Delta t, \quad k, l = 1, \dots, K.\]
We then use spectral decomposition to expand the solution and source term in terms of a set of orthogonal eigenfunctions, projecting the original PDE onto these eigenmodes
\[ u(x,t) = \sum_{k=1}^{K} u_k(t) \varphi_k(x), \quad f(x) = \sum_{k=1}^{K} f_k \varphi_k(x), \quad g(x) = \sum_{k=1}^{K} g_k \varphi_k(x).
\]
When calculating the modal coefficients of the true value source terms, we need to perform numerical integration on $f_k$, $g_k$, and $u_k(T)$\\
\[
f_k = \sum_{i=0}^{N_x} f_{\text{true}}(x_i) \varphi_k(x_i) \Delta x, \quad g_k = \sum_{i=0}^{N_x} g_{\text{true}}(x_i) \varphi_k(x_i) \Delta x.
\]
\[
u_k(T) = (f_k + b g_k) \int_0^T A_k(\tau) \, d\tau + \sigma g_k \int_0^T A_k(\tau) \, dW(\tau) + g_k \sum_{s_j \leq T} A_k(s_j) J_j.
\]
The selection rules for the parameters of the finite Lévy jump process are: a) The number of jumps $N_J^{(n)} \sim \text{Poisson}(\lambda T)$. b) The jump times $\{s_j^{(n)}\}_{j=1}^{N_J^{(n)}}$ are uniformly sampled from $[0, T]$. c) The jump magnitudes $\{J_j^{(n)}\}_{j=1}^{N_J^{(n)}} \sim \mathcal{N}(0, \ sigma_J^2)$. The Brownian increments satisfy: $\text{For each time step } j = 1, \dots, N_t, \, \Delta W_j \sim \mathcal{N}(0, \Delta t). $
For the discrete treatment of each modal contribution in $u_k^{(n)}$, the deterministic time integral is discretized using the equidistant rectangular method, the random Brownian term is discretized using the classical Euler–Maruyama method, and the jump term is directly calculated using the continuous jump time obtained from sampling to compute $A_k(s_m)$
\[
u_k^{(n)} \approx \left( f_k + b g_k \right) \sum_{j=0}^{N_t-1} A_k(t_j) \Delta t + \sigma g_k \sum_{j=1}^{N_t} A_k(t_{j-1}) \Delta W_j + g_k \sum_{m=1}^{N_J^{(n)}} A_k( s_m) J_m.
\]
Next, spatial reconstruction and noise addition are performed, with noise selected as the commonly used Gaussian noise
\[
u_k^{(n)}(x_i, T) = \sum_{k=1}^{K} u_k^{(n)} \varphi_k(x_i), \quad u_{\text{obs}}(x_i) = u^{(n)}(x_i, T) + \epsilon_i^{(n)}, \quad \epsilon_i^{(n)} \sim \mathcal{N}(0, \sigma_\epsilon^2).
\]
Then, the overall mean and covariance are obtained through data statistics and projection.
\begin{itemize}[left=5pt]
	\item \textbf{Modal projection}
	\[
	U_k^{(n)} = \sum_{i=0}^{N_x} u_{\text{obs}}^{(n)}(x_i) \varphi_k(x_i) \Delta x, \quad k = 1, \dots, K, \quad n = 1, \dots, N_{\text{samples}}.
	\]
	
	\item \textbf{Sample mean}
	\[
	\bar{U}_k = \frac{1}{N_{\text{samples}}} \sum_{n=1}^{N_{\text{samples}}} U_k^{(n)}.
	\]
	
	\item \textbf{Sample covariance}
	\[
	C_{kl} = \frac{1}{N_{\text{samples}} - 1} \sum_{n=1}^{N_{\text{samples}}} (U_k^{(n)} - \bar{U}_k)(U_l^{(n)} - \bar{U}_l), \quad k, l = 1, \dots, K.
	\]
\end{itemize}

The handling of the reverse source problem is mainly achieved by reconstructing $f$ and $g$ using $(\ref{f})$ and $(\ref{g})$\\
\textbf{Deterministic source $f$}
The theoretical relationship
\[
\bar{U}_k = f_k = \sum_{j=0}^{N_t-1} A_k(t_j) \Delta t \quad \Rightarrow \quad f_k^{\text{rec}} = \bar{U}_k = \sum_{j=0}^{N_t-1} A_k(t_j) \Delta t.
\]
Spatial reconstruction formula
\[
f_{\text{rec}}(x_i) = \sum_{k=1}^{K} f_k^{\text{rec}} \varphi_k(x_i).
\]
\textbf{Random source $g$}
The model covariance is discretized as
\[
\hat{C}_{kl}(g) = \gamma g_k g_l I_{kl} + \sigma_\epsilon^2 \delta_{kl}.
\]
The objective function is expressed as follows
\[
J(g_1, \dots, g_K) = \sum_{k=1}^{K} \sum_{l=1}^{K} (C_{kl} - \hat{C}_{kl}(g))^2 + \alpha \sum_{k=1}^{K} g_k^2.
\]
where $\alpha$ is the regularization parameter.
Optimization yields $\{g_k^{\text{rec}}\}$, so that
\[
g_{\text{rec}}(x_i) = \sum_{k=1}^{K} g_k^{\text{rec}} \varphi_k(x_i).
\]
Finally, the $L^2$ error estimate is given by the following formula
\[
\epsilon_f = \frac{\sqrt{\sum_{i=0}^{N_x} \left(f_{\text{true}}(x_i) - f_{\text{rec}}(x_i)\right)^2 \Delta x}}{\sqrt{\sum_{i=0}^{N_x} f_{\text{true}}(x_i)^2 \Delta x}},
\]
\[
\epsilon_g = \frac{\sqrt{\sum_{i=0}^{N_x} \left(g_{\text{true}}(x_i) - g_{\text{rec}}(x_i)\right)^2 \Delta x}}{\sqrt{\sum_{i=0}^{N_x} g_{\text{true}}(x_i)^2 \Delta x}}.
\]

\begin{remark}
	In the random source term reconstruction process, we reconstruct $g$ rather than $g^2$. It should be noted that since the covariance of the observed data is insensitive to the sign of $g(x)$, the sign of its modal coefficient $g_k$ is ambiguous. Therefore, we perform phase correction in the code to obtain results that are more consistent with the actual situation, but this usually requires additional prior information. However, in practical applications, since the true sign is unknown, such direct correction is often infeasible, necessitating the reconstruction of $g^2$ or the use of other methods to address the issue of sign uncertainty.  
\end{remark}

\subsubsection{Numerical experiment results}
For simple case, we can assume that the true values are $f_{true}(x)=\sin x$ and $g_{true}(x)=\sin x$.
Under two noise levels, $\sigma=0.001$ and $\sigma=0.005$, the reconstruction results are as follows.
\begin{figure}[h]  
	\centering  
	\includegraphics[width=0.7\textwidth]{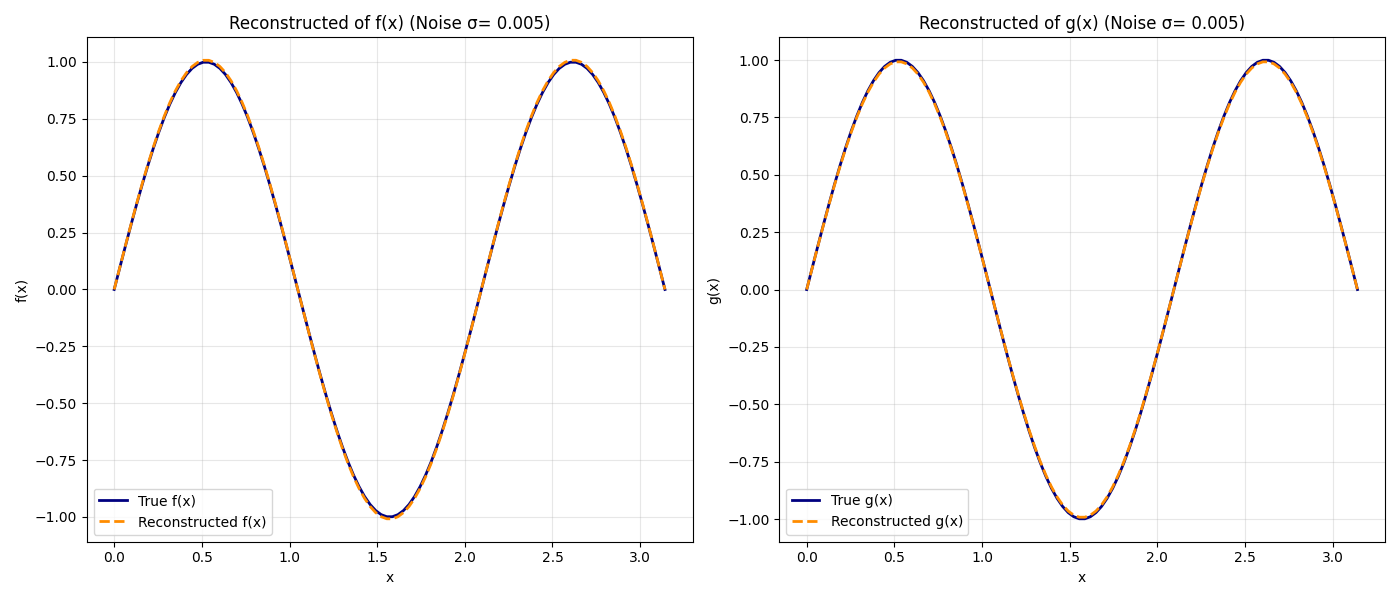}  
	\caption{Reconstructed effect with noise $\sigma=0.005$}  
	\label{fig3}  
\end{figure}
\newpage
\begin{figure}[h]  
	\centering  
	\includegraphics[width=0.7\textwidth]{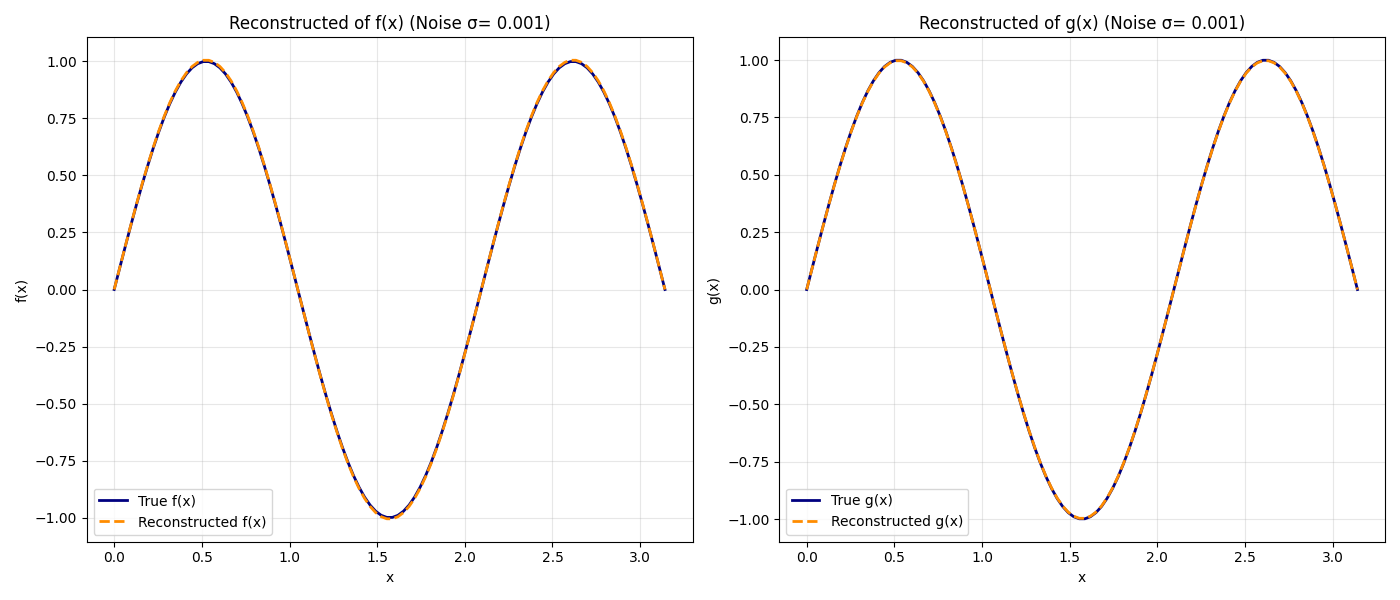}  
	\caption{Reconstructed effect with noise $\sigma=0.001$}  
	\label{fig4}  
\end{figure}

\begin{table}[h!]
	\centering  
	\begin{tabular}{c|c|c}  
		\hline  
		$\sigma$ & 0.001 & 0.005 \\  
		\hline  
		$f$ & 0.0060 & 0.0088 \\  
		\hline  
		$g$ & 0.0022 & 0.0069 \\  
		\hline  
	\end{tabular}
	\caption{Relative $L^2$ error}  
	\label{b1}  
\end{table}

More complicated, we assume the true values are $f_{true}(x)=\sin x$ and $g_{true}(x)=\exp\left( -\left(x - 0.5L\right)^2 \right)$.
Under two noise levels, $\sigma=0.001$ and $\sigma=0.005$, the reconstruction results see Figure \ref{fig5} and Figure \ref{fig6}.

Overall, the results show that the reconstruction under a single Fourier mode with true values provides good characterization of $f$ and $g$. As shown in Table \ref{b1}, the error estimates reach $1e-3$ under both noise levels. In the reconstruction under multiple Fourier modes, as shown in Table \ref{b2}, the error estimate is only $1e-2$, especially for the reconstruction of $g$. It can be seen that as the number of Fourier modes increases, the reconstruction process becomes unstable.
\newpage
\begin{figure}[h]  
	\centering  
	\includegraphics[width=0.7\textwidth]{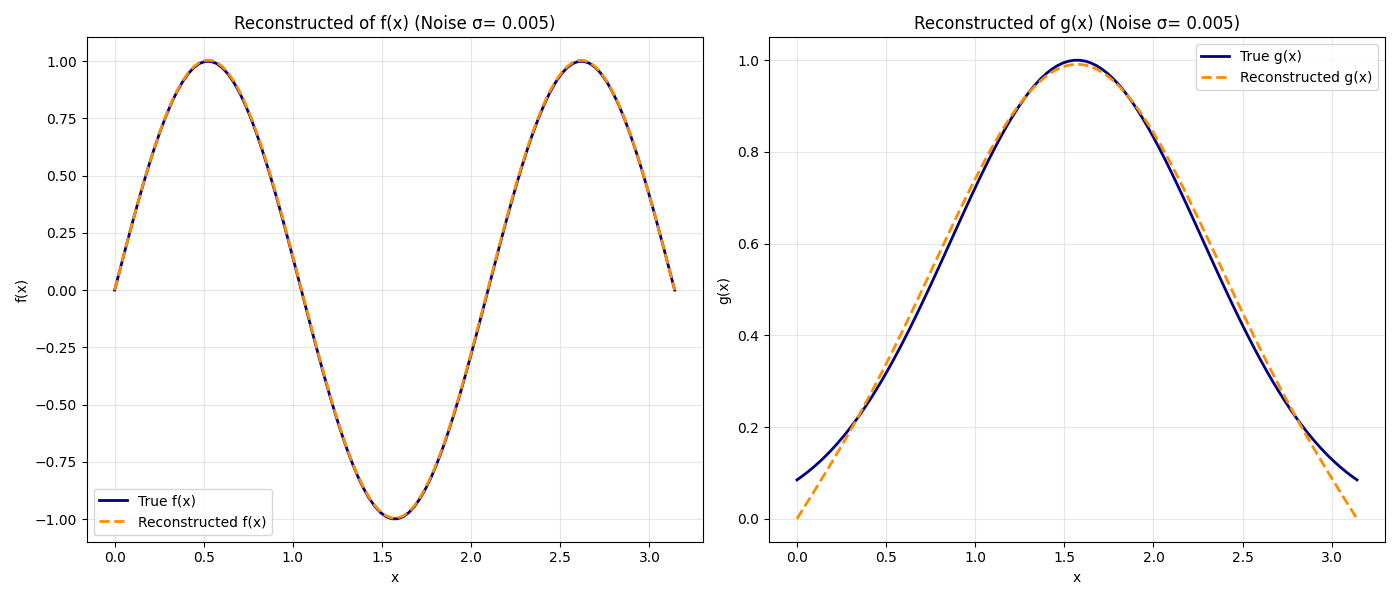}  
	\caption{Reconstructed effect with noise $\sigma=0.005$}  
	\label{fig5}  
\end{figure}
\begin{figure}[h]  
	\centering  
	\includegraphics[width=0.7\textwidth]{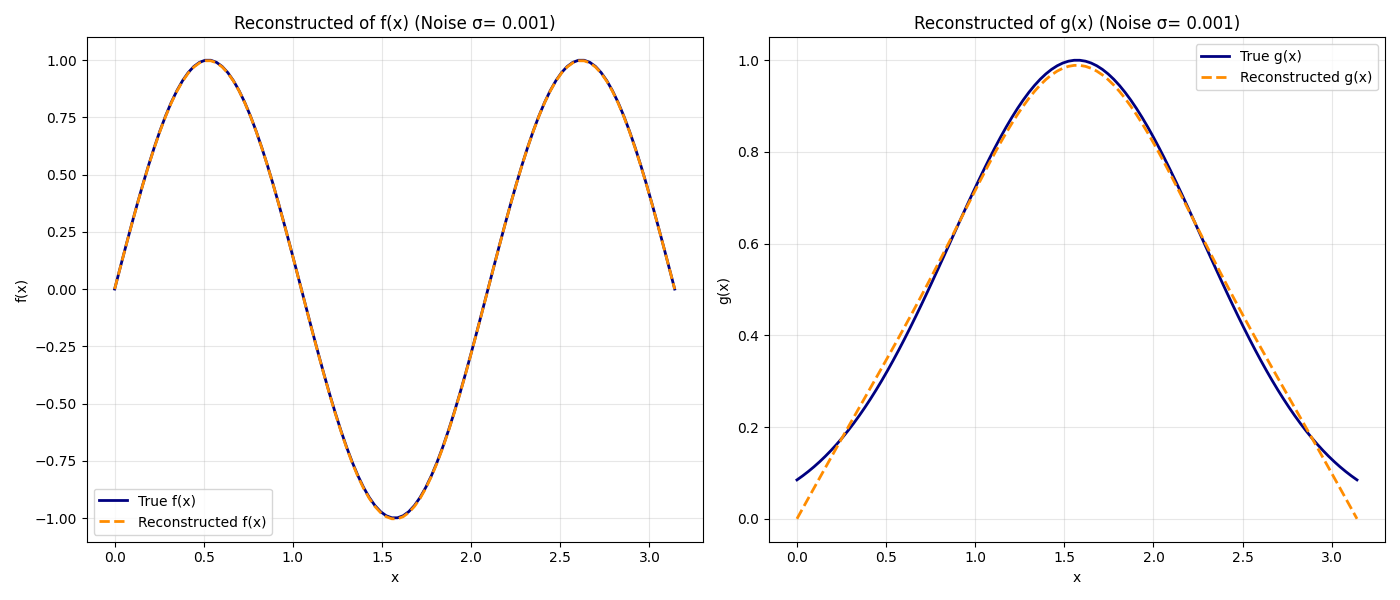}  
	\caption{Reconstructed effect with noise $\sigma=0.001$}  
	\label{fig6}  
\end{figure}
\begin{table}[h!]
	\centering  
	\begin{tabular}{c|c|c}  
		\hline  
		$\sigma$ & 0.001 & 0.005 \\  
		\hline  
		$f$ & 0.0033 & 0.0064 \\  
		\hline  
		$g$ & 0.0370 & 0.0410 \\  
		\hline  
	\end{tabular}
	\caption{Relative $L^2$ error}  
	\label{b2}  
\end{table}

\section{Conclusion}
This paper investigates the inverse source problem for the one-dimensional stochastic Helmholtz equation without attenuation and the stochastic wave equation driven by a finite-jump Lévy process. For the ill-posedness of the inverse problem of the stochastic Helmholtz equation without attenuation, a new computational method is proposed to significantly alleviate the difficulty of reconstructing the strength; for the stochastic wave equation driven by a finite-jump Lévy process, a stability estimate for the direct problem is obtained. After analyzing the instability of the reconstruction process in the inverse problem, a method is proposed to reconstruct the source function using the data of the wave field at the final time point $u(x,T)$. The reconstruction of the source function in the inverse problems of the two types of equations is achieved through two contrasting methods: one is a point-by-point recovery process combined with a multi-frequency fusion regularization method, and the other is a recovery process based on a global modal spectral decomposition method. Future research will consider the well-posedness problem of the inverse source problem for stochastic wave equations driven by infinite-jump Lévy processes.
\section*{Acknowledgments}
The work was supported by NSFC Project (12431014). We would like to express our sincere gratitude to Professor Li Peijun for his professional guidance in research ideas, argumentation analysis, and paper writing, which enabled us to  complete the paper successfully.

\bibliographystyle{plain}


\end{document}